\documentclass[reqno]{amsart}
\usepackage{amssymb}
\usepackage{mathrsfs}
\usepackage[stretch=10,shrink=10]{microtype}
\usepackage{tikz}
\usetikzlibrary{arrows.meta}
\usepackage{mathtools}
\usepackage[showframe=false, margin = 1.5 in]{geometry}
\usepackage[shortlabels]{enumitem}
\usepackage{color}
\usepackage{graphicx}
\usepackage{bm} 
\usepackage{hyperref}

\newtheorem{thm}{Theorem}[section]
\newtheorem{lemma}[thm]{Lemma}
\newtheorem{prop}[thm]{Proposition}
\newtheorem{cor}[thm]{Corollary}

\theoremstyle{remark}
\newtheorem{remark}[thm]{Remark}

\theoremstyle{definition}
\newtheorem{question}[thm]{Question}

\newtheorem{defines}[thm]{Definitions}

\newtheorem{example}[thm]{Example}
\newcommand{\E}{\mathbf{E}}
\renewcommand{\P}{\mathbf{P}}
\newcommand{\1}{\mathbf{1}}
\newcommand{\Poi}{\mathrm{Poi}}
\newcommand{\Ber}{\mathrm{Bernoulli}}

\newcommand{\Bb}{\mathcal{B}}

\newcommand{\Gg}{\mathcal{G}}
\newcommand{\Ggg}{\mathscr{G}}
\newcommand{\Ww}{\mathcal{W}}

\renewcommand{\emptyset}{\varnothing}
\newcommand{\Tt}{\mathcal{T}}
\newcommand{\inftrees}{\Tt_{\text{inf}}}
\newcommand{\lambdacrit}{\lambda_{\text{crit}}}

\newcommand{\NN}{\mathbb{N}}
\renewcommand{\1}{\mathbf{1}}
\newcommand{\bigmid}{\;\big\vert\;}
\newcommand{\Bigmid}{\;\Big\vert\;}
\newcommand{\GW}{\mathrm{GW}}
\newcommand{\Tpiv}{T_{\mathrm{piv}}}

\DeclarePairedDelimiter\abs{\lvert}{\rvert}%
\newcommand{\Imax}{I^{\max}}
\newcommand{\Ff}{\mathcal{F}}
\newcommand{\Cc}{\mathcal{C}}
\newcommand{\Fff}{\mathscr{F}}
\newcommand{\Aa}{\mathcal{A}}
\newcommand{\zs}{{0\mathrm{s}}}
\newcommand{\zd}{{0\mathrm{d}}}
\newcommand{\os}{{1\mathrm{s}}}
\newcommand{\od}{{1\mathrm{d}}}
\newcommand{\Ts}{\mathcal{T}_{\mathrm{s}*}}
\newcommand{\Tcol}{\mathcal{T}_{\mathrm{col}}}
\newcommand{\Tcolstar}{\mathcal{T}_{\mathrm{col}*}}
\newcommand{\chicol}{\chi_{\mathrm{col}}}
\newcommand{\chicolstar}{\chi_{\mathrm{col}*}}
\newcommand{\Ttaut}{\Tt^{\mathrm{taut}}}
\newcommand{\RST}{\mathrm{RST}}
\newcommand{\Tspine}{T^{\text{spine}}}

\author{Tobias Johnson}
\address{Department of Mathematics, College of Staten Island}
\email{tobias.johnson@csi.cuny.edu}

\author{Moumanti Podder}
\address{NYU-ECNU Institute of Mathematical Sciences, New York University, Shanghai}
\email{mpodder3@math.gatech.edu}

\author{Fiona Skerman}
\address{Department of Informatics, Masaryk University}
\email{skerman@fi.muni.cz}
\thanks{T.J.\ received support from NSF grants DMS-1401479 and DMS-1811952. 
M.P. acknowledges 
partial support from NSF CAREER grant CCF:AF-1553354.}

\title[Random tree recursions]{Random tree recursions: which fixed points correspond to tangible sets of trees?}



\begin{document}
\bibliographystyle{amsplain}

\begin{abstract}
  Let $\Bb$ be the set of rooted trees containing an infinite binary subtree starting at the root.
  This set satisfies the metaproperty that a tree belongs to it if and only if 
  its root has children $u$ and $v$ such that the subtrees rooted at $u$ and $v$ belong to it.
  Let $p$ be the probability that a Galton--Watson tree falls in $\Bb$. 
  The metaproperty makes $p$ satisfy a fixed-point equation, 
  which can have multiple solutions. One of these solutions is $p$, 
  but what is the meaning of the others? In particular, are they probabilities
  of the Galton--Watson tree falling into other sets satisfying the same metaproperty?
  We create a framework for posing questions of this sort, and we classify solutions to fixed-point
  equations according to whether they admit probabilistic interpretations.
  Our proofs use spine decompositions of Galton--Watson trees and the analysis of Boolean functions.
\end{abstract}

\maketitle

\section{Introduction}

A seminal problem in discrete probability is to determine the probability of survival
of a Galton--Watson tree. For the sake of simplicity, suppose that the offspring distribution is
$\Poi(\lambda)$, and denote the tree by $T_\lambda$. Let $\inftrees$ denote the set of infinite rooted trees. Let $p$ denote the survival
probability, given by $\P[T_\lambda\in\inftrees]$.
The typical solution gives $p$ as a fixed point of a map as follows: 
Let $Z$ be the number of children $v$ of the root of $T_\lambda$ such that the subtree rooted at $v$ is infinite. Each subtree is infinite with probability~$p$, just like the original tree. Thus
$Z \sim \Poi(p \lambda)$ by Poisson thinning. Since $T_{\lambda}$ is infinite if and only if $Z \geq 1$,
\begin{align}\label{eq:infinite.recurrence}
  p &= 1 - e^{-\lambda p}.
\end{align}
As is well known (see \cite{athreya-ney}), when $\lambda>1$, this equation has two solutions, and the positive one is the true value of $p$.
In arriving at \eqref{eq:infinite.recurrence}, the only property of $\inftrees$ we used was that $t\in\inftrees$ if and only if there exists some child $v$ of the root of $t$ such that the subtree descending from $v$
is in $\inftrees$. Let us call this the metaproperty of $\inftrees$ that yields
\eqref{eq:infinite.recurrence}.

Again, let $T_\lambda$ be a Galton--Watson tree with child distribution $\Poi(\lambda)$.
It is natural to ask if there is some other set of trees $\Tt_0$ satisfying the metaproperty
such that $\P[T_\lambda\in\Tt_0]$ is the other solution to \eqref{eq:infinite.recurrence}, which is $0$. 
A bit of thought reveals that $\Tt_0=\emptyset$ fits this criteria. Vacuously, $t\in\emptyset$
if and only if the root of $t$ has a child whose subtree is in $\emptyset$, and clearly $\P[T_\lambda\in\emptyset]=0$. Thus, the metaproperty yields an equation with two solutions, and each solution gives the probability
under the Galton--Watson measure of a set of trees satisfying the metaproperty.
Indeed, we will later see that $\inftrees$ and $\emptyset$ are the \emph{only} two sets of trees
satisfying the metaproperty, up to measure zero changes under the Galton--Watson measure with
child distribution $\Poi(\lambda)$ (see Remark~\ref{rmk:measure.zero} for more discussion
on measure zero changes).

This work was motivated by a nearly identical example that is more difficult to resolve. This time, we consider sets of trees $\Bb$ where $t\in\Bb$ if and only if the root of $t$ has at least \emph{two} children $u$ and $v$ whose subtrees are in $\Bb$. Let us call this metaproperty the \emph{at-least-two rule}. 
Suppose $p=\P[T_{\lambda}\in\Bb]$ for some set of trees $\Bb$ obeying the at-least-two rule. Invoking Poisson thinning and self-similarity as in the first example, we get
\begin{align}\label{eq:at.least.two.recurrence}
  p &= 1 - e^{-\lambda p}(1+\lambda p).
\end{align}
As explained in \cite{PSW}, which investigated the existence of a giant $3$-core in a
random graph, there is a critical parameter $\lambdacrit\approx 3.35$
where this equation changes behaviour (see Figure~\ref{fig:at.least.two.phases}).
\begin{figure}
\centering
    \begin{tikzpicture}[yscale=2.75]
      \draw (0,0) -- (0,1);
      \foreach \x in {.2, .4, .6, .8, 1}
 \draw (.1, \x)--(0,\x) node[left,font=\small] {$\x$};
      \foreach \x in {1, 2, 3, 4, 5}
\draw (\x, .03636)--(\x, 0) node[below,font=\small] {$\x$};
      \draw (0,0) node [left,font=\small] {$0$};
      \path (-.8, .5) node {$p$}
            (2.5, -.25) node {$\lambda$};
      \draw (0,0) -- (5,0);
      \draw[very thick,green] plot file{iota2_upper.table};
      \draw[very thick,blue] plot file{iota2_lower.table};
      \draw[very thick,red] (0,0)--(5,0);                                                
    \end{tikzpicture}
\label{fig:at.least.two.phases}
\caption{A plot showing all $p$ satisfying \eqref{eq:at.least.two.recurrence} for given $\lambda$.
For $\lambda<\lambdacrit\approx 3.35$, the only solution to \eqref{eq:at.least.two.recurrence}
is $p=0$. For $\lambda=\lambdacrit$, there are two solutions for $p$, and for $\lambda>\lambdacrit$,
there are three.}
\end{figure}
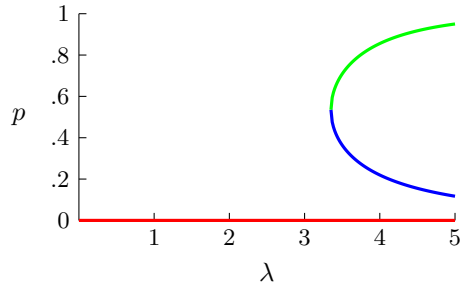
For all $\lambda>0$, there is a trivial solution to \eqref{eq:at.least.two.recurrence} given by $p=0$.
When $\lambda<\lambdacrit$, this is the only solution.
At $\lambda=\lambdacrit$, a second solution emerges, and when $\lambda>\lambdacrit$ there are three
solutions. (We prove these statements in Example~\ref{ex:answer}.)
Let $\Bb_0$ be the set of all trees that contain an infinite binary subtree starting at the root.
Note that $\Bb_0$ satisfies the at-least-two rule. 
It was shown by Dekking \cite{Dekking} (also see \cite{PakesDekking}) that $\P[T_\lambda\in\Bb_0]$ is the largest solution to \eqref{eq:at.least.two.recurrence} when $\lambda>\lambdacrit$, shown in green in Figure~\ref{fig:at.least.two.phases}.
An immediate intuition as to why the green curve is the one corresponding to $\P[T_\lambda\in\Bb_0]$ is that this is the only curve which is increasing in $\lambda$. Another set of trees obeying the at-least-two rule
 is the empty set. Obviously, $\P[T_\lambda\in\emptyset]=0$, the smallest solution to \eqref{eq:at.least.two.recurrence}, shown in red in Figure~\ref{fig:at.least.two.phases}. 
 Joel Spencer posed the question that set this work in motion:
 is there a set of trees to go with the middle solution (shown in blue in Figure \ref{fig:at.least.two.phases})? More formally, the question asks the following:
\begin{question}[Spencer]\label{q:spencer}
  Let $T_\lambda$ be a Galton--Watson tree with child distribution $\Poi(\lambda)$.
Say that a set of trees $\Bb$ follows the at-least-two rule if $t\in\Bb$ if and only if the root of $t$ has two children $u$ and $v$ such that the subtrees rooted at $u$ and $v$ are also in $\Bb$. Suppose that $\lambda>\lambdacrit$. Does there exist a set of trees $\Bb$ following the at-least-two rule such that $\P[T_\lambda\in\Bb]$ is the middle solution of \eqref{eq:at.least.two.recurrence}?
\end{question}


We answer this question in the negative. More generally,
our main result, Theorem~\ref{main 2 colours}, gives the answer to \emph{any} question of this form. In the language of this paper, it is a criterion for which fixed points of tree automata admit interpretations. In this example, the tree automaton refers to the at-least-two
metaproperty. For the Galton--Watson child distribution $\Poi(\lambda)$,
the fixed points of this automaton are the solutions to \eqref{eq:at.least.two.recurrence}. An interpretation corresponds to a set of trees following the metaproperty given by the automaton. Theorem~\ref{main 2 colours} shows that $\Bb_0$ and $\emptyset$ are the only two sets of trees following the at-least-two rule, up to measure zero changes under the Galton--Watson measure with child distribution $\Poi(\lambda)$.
\begin{remark}\label{rmk:measure.zero}
  It is important that we consider sets of trees satisfying a metaproperty only up to measure
  zero changes under a Galton--Watson measure with a given child distribution.
  For example, let $\Tt$ be the set of trees that contain an infinite binary subtree
  somewhere within them (i.e., not necessarily starting from the root).
  This set satisfies our original metaproperty: a tree is in $\Tt$ if and only if its root has
  at least one child initiating a tree in $\Tt$. 
  But on its face, $\Tt$ is neither $\inftrees$ nor $\emptyset$,
  which we claimed were the only sets of trees satisfying this metaproperty. The solution
  to this apparent paradox is that from the perspective of the Galton--Watson tree $T_\lambda$
  with child distribution $\Poi(\lambda)$, the set $\Tt$ is in fact equivalent to either
  $\inftrees$ or $\emptyset$. For $\lambda<\lambdacrit$, there is zero probability that
  $T_\lambda$ lies in $\Tt$, and hence $\Tt$ is a measure zero change away from $\emptyset$.
  For $\lambda\geq\lambdacrit$, the tree $T_\lambda$ falls in $\Tt$ with probability~$1$ given that
  $T_\lambda$ is infinite. Hence $\Tt$ is a measure zero change away from $\inftrees$ in this case.
\end{remark}

\subsection{Summary of main result}
We start by giving a nonrigorous version of our main result, since
it will take some effort to state all the definitions we need for a formal statement.
A \emph{tree automaton} is a set of rules determining
the colour of a parent vertex in a tree from the colour of its children. 
Let $\Sigma$ be a finite set representing the possible colours. The automaton
corresponding to the at-least-two rule acts on colours $\Sigma=\{0,1\}$, 
assigning colour~$1$ to the parent if and
only if it has at least two children of colour~$1$. A \emph{fixed point} of a tree automaton
is a probability distribution $\vec{\nu}$ on $\Sigma$ such that if a Galton--Watson tree is generated
and the children of the root are assigned i.i.d.-$\vec{\nu}$ colours, then the colour of the root
induced by the automaton is also distributed as $\vec{\nu}$. For the example presented earlier,
the fixed points have the form $\Ber(p)$, where $p$ satisfies \eqref{eq:at.least.two.recurrence}.
To define an \emph{interpretation} of a tree automaton, suppose we have a map $\iota\colon\Tt\to\Sigma$,
where $\Tt$ is the space of rooted trees. Now, imagine colouring each vertex $v$
in an arbitrary tree by applying $\iota$ to the subtree rooted at $v$. If the resulting
colouring of the tree is always consistent with the rules given by the tree automaton, then we call
the map an interpretation of the automaton. We saw two interpretations in our earlier example: 
the first mapped a tree to $0$ or $1$ depending on whether it contained an infinite binary tree
starting at its root, and the second mapped all trees to $0$.

It is not hard to see that the colour of a Galton--Watson tree assigned by an interpretation
of an automaton must be distributed as a fixed point of the automaton (see
Lemma~\ref{iota distribution}). For example, if $\iota$
is the first interpretation described above and $T$ is a Galton--Watson tree with
child distribution $\Poi(\lambda)$, then $\iota(T)$ is distributed as $\Ber(p_2)$, where $p_2$ is the largest
solution to \eqref{eq:at.least.two.recurrence}. Our main result flips this around,
letting us determine for a given fixed point $\vec{\nu}$
whether there exists an interpretation $\iota$ such that $\iota(T)\sim\vec\nu$.

The criterion is based on an object we call the \emph{pivot tree}. Essentially, first
generate the Galton--Watson tree to level~$n$. Then, randomly colour the vertices at level~$n$
by sampling independently from the given fixed point. Apply the automaton
to colour the vertices at levels $0$ to $n-1$. Now, call a vertex \emph{pivotal} for this colouring
if altering its colour and recolouring all of its ancestors by the automaton
alters the colour of the root (see Figure~\ref{fig:pivotal.two.state}). 
The set of all pivotal vertices to level~$n$ then
forms a random subtree of the original Galton--Watson tree. There is a natural way to extend this construction
beyond a fixed $n$ to give a (possibly) infinite tree, the \emph{pivot tree}, which turns out
to be multitype Galton--Watson.
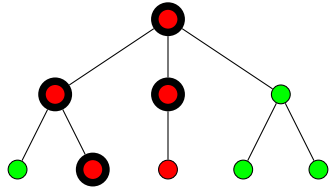
\begin{figure}
  \begin{tikzpicture}[zero/.style={circle,draw=black,fill=red,inner sep=0,
    minimum size=0.25cm}, 
      bzero/.style={circle,draw=black, fill=red, inner sep=0, minimum size=0.35cm,line width=.1cm},
      unknown/.style={circle,draw=black,fill=gray,inner sep=0,
    minimum size=0.25cm},
      one/.style={circle,draw=black,fill=green,inner sep=0,
        minimum size=0.25cm}]
    \path (0,0) node[bzero] (R) {};
    \path (R) + (-1.5,-1) node[bzero] (R1) {}    
                              + (0,-1) node[bzero] (R2) {};
    \path (R)+ (1.5,-1) node[one] (R3) {};
    \path (R1) + (-.5,-1) node[one] (R11) {}
          (R2) + (0,-1) node[zero] (R21) {}
          (R3) + (-.5,-1) node[one] (R31) {}
          (R3) + (.5,-1) node[one] (R32) {};
    \path (R1) + (.5,-1) node[bzero] (R12) {};
    \path[draw] 
        (R)--(R1) (R)--(R2) (R)--(R3)
        (R1)--(R11) (R1)--(R12) (R2)--(R21) (R3)--(R31) (R3)--(R32);    
    \end{tikzpicture}
  \caption{The first three levels of a tree coloured consistently with the at-least-two automaton
    given in Example~\ref{ex:at.least.two}. Red denotes state~$0$ and green denotes state~$1$.
    Vertices in bold are pivotal, meaning that flipping their colours and recolouring above them
    according to the automaton
    causes the root to flip colours.}
  \label{fig:pivotal.two.state}
\end{figure}

Loosely speaking, the main result of this paper is that when $\abs{\Sigma}=2$,
a given fixed point of a tree automaton
has a corresponding interpretation if and only if the associated pivot tree is subcritical or critical
(or equivalently, if it is almost surely finite).
If so, then it has precisely one interpretation, up to measure zero changes with respect to the
Galton--Watson measure. This criterion is quite practical to check, 
and we do so for the at-least-two automaton
and some other examples in Section~\ref{sec:examples}.

When $3\leq\abs{\Sigma}<\infty$, we prove only that a subcritical pivot tree implies existence of
an interpretation. We believe that our approach in this paper can be adapted to prove that a supercritical
pivot tree implies nonexistence of an interpretation, but there are several complications
(see Remark~\ref{rmk:kstate.supercritical}).

We now proceed to define these terms more formally. We then state our main results
in Section~\ref{sec:main.results}.

\subsection{Notation} \label{notations}
We define $\mathcal{T}$ to be the set of locally finite, ordered, rooted trees (\emph{ordered}
means that an ordering is given for the children of each vertex).
This set can be viewed as a metric space (see \cite[Exercise~5.2]{LP}), which we endow
with its Borel $\sigma$-algebra to make a measure space.
Our results will be for Galton--Watson trees with general child distributions, sometimes
under mild moment conditions. We will typically
denote the tree by $T$ and the child distribution by $\chi$.
We always assume that $\chi$ puts positive probability on $\{2,3,\ldots\}$,
so that $T$ is a true tree.
For any tree $t\in\Tt$, we let $V(t)$ denote its vertex set and $R_{t}$ its root. 
Let $t(v)$ denote the subtree of $t$ made up of $v$ and its descendants. We let $t|_{n}$ denote the tree obtained by truncating $t$ beyond its $n$th generation and $[t]_n\subseteq\Tt$ the set of trees
that match $t$ up to the $n$th generation, where the root is considered to belong to generation~$0$.
Let $L_{n}(t)$ denote the set of all nodes of $t$ in generation~$n$, and let $\ell_n(t)=\abs{L_n(t)}$.
We abbreviate $L_{n}(T)$ by $L_{n}$ and $\ell_n(T)$ by $\ell_n$.

We will often work with \emph{coloured trees}, defined as a pair $(t,\tau)$ consisting
of a tree $t\in\Tt$ together with a colouring $\tau\colon V(t)\to\Sigma$.
We denote the space of coloured trees as $\Tcol$, taking the set of colours $\Sigma$
as fixed in advance. For $(t,\tau)\in\Tcol$, let $[t,\tau]_n\subseteq\Tcol$ denote the set
of coloured trees that match $(t,\tau)$ up to the $n$th generation.

\subsection{Tree automata}\label{tree automata}
Let $\Sigma$ denote a finite set, to be thought of as colours or states. A \emph{tree automaton} on the states $\Sigma$ is essentially a set of rules for determining the state of a parent in the tree from the states of its children. Formally, we define an automaton as a map $A\colon\NN_0^{\Sigma}\to\Sigma$, where $\NN_0=\NN\cup\{0\}$. The vector $\vec{n} = \big(n_{\sigma}: \sigma \in \Sigma\big) \in \NN_0^\Sigma$ represents the count of children in each state, and $A(\vec{n})$ represents the state assigned to the parent. 

\begin{example}[At-least-two automaton]\label{ex:at.least.two}
  We define an automaton $A$ on states $\Sigma=\{0,1\}$ that 
  assigns state~$1$ to the parent if and only if at least two of 
  its children have state~$1$. Formally, the automaton is
  the map $(n_0,n_1)\mapsto\1\{n_1\geq 2\}$.
  As we mentioned, this automaton is implicit in Question~\ref{q:spencer}.
\end{example}

Tree automata are of interest in logic and theoretical computer science. In these settings, they typically
act on trees with vertex labels rather than plain trees, and there are some restrictions on them.
See \cite{automata_1} and \cite[Chapter~7]{Libkin} for more details on automata for finite trees, and
\cite[Section~6]{Thomas} for more on infinite trees.
Tree automata can be used to determine
which sets of trees can be defined by a given logic. For example, call a set of trees \emph{regular}
if there exists a tree automaton so that a tree falls into the set if and only if the automaton
assigns its root one of a set of accepted states. A set of finite trees is definable in
monadic second-order logic if and only if it is regular \cite[Theorem~7.30, Theorem~7.34]{Libkin}.
A similar statement holds for infinite trees as well \cite[Theorem~6.19]{Thomas}. 
We will revisit logic in Section~\ref{sec:connections}, after we state our results.

For a given tree $t$, we say that an assignment of colours $\tau\colon V(t)\to\Sigma$ is \emph{compatible with the automaton~$A$} if for every $v \in V(t)$, we have 
\begin{equation}\label{compatible}
\tau(v) = A(\vec{n}), 
\end{equation}
where $\vec{n} = (n_{\sigma}\colon \sigma \in \Sigma)$ and $n_{\sigma}$ is the number of children of $v$ that are coloured $\sigma$ under $\tau$. If $t$ is finite, there is only one colouring compatible with $A$. At each leaf, this colouring takes the value $A(0,\ldots,0)$, and then the automaton determines the colours of all other vertices. When $t$ is infinite, however, there are typically many assignments compatible with a given automaton.

\subsection{Interpretations}\label{interpretations}

An interpretation of an automaton is a deterministic classification of trees into the states of $\Sigma$ such that the state of a tree can be computed from the states of the subtrees descending from the children of its root, according to the rules of the automaton. For example, assign a tree state~$1$ if it contains an infinite binary subtree starting at its root, and assign it state~$0$ otherwise. This is an interpretation of the at-least-two automaton of Example~\ref{ex:at.least.two}, since a tree $t$ has state~$1$ if and only if its root $R_{t}$ has at least two children $u, v$ with subtrees $t(u), t(v)$ in state~$1$.

Formally, we define an interpretation as follows. Let $\chi$ be a probability measure on the nonnegative integers, and let $\GW(\chi)$ denote the Galton--Watson measure on $\Tt$ with child distribution $\chi$. We call a
measurable map $\iota \colon \mathcal{T} \rightarrow \Sigma$ an \emph{interpretation of the automaton~$A$ under $\GW(\chi)$}, if for a.e.-$\GW(\chi)$ tree $t\in\Tt$, the colouring $\tau\colon V(t) \rightarrow \Sigma$, defined as 
\begin{equation}\label{interpretation}
\tau(v) = \iota(t(v)), \text{ for all } v \in V(t),
\end{equation}
is compatible with $A$. Typically, we will call $\iota$ an interpretation of $A$ without mentioning $\GW(\chi)$, since the offspring distribution will be fixed throughout. For many interpretations, including our example of assigning a tree $1$ if it contains an infinite binary subtree from the root, the compatibility condition holds for every tree in $\Tt$, and the measure $\GW(\chi)$ is irrelevant.

\subsection{Fixed points and their connections with interpretations}\label{fixed points}
Let $T\sim\GW(\chi)$. If $\iota\colon\Tt\to\Sigma$ is an interpretation of the automaton~$A$ under $\GW(\chi)$, then the distribution of $\iota(T)$ is constrained by the self-similarity of $T$. For example, if $\chi\sim\Poi(\lambda)$ and $\iota$ is an interpretation of the at-least-two automaton of Example~\ref{ex:at.least.two}, then $\iota(T)\sim\Ber(p)$, where $p$ satisfies \eqref{eq:at.least.two.recurrence}.

We now describe these constraints on the distribution of $\iota(T)$ when $\iota$ is an interpretation of a general automaton $A$ and $T\sim\GW(\chi)$. Let $D$ denote the set of all probability distributions on $\Sigma$ (as $\Sigma$ is finite, $D$ is a finite-dimensional simplex). We define a map $\Psi\colon D \rightarrow D$ that we call the \emph{automaton distributional map corresponding to $A$ and $\chi$}, as follows. Fix $\vec{x} \in D$. 
Consider a random tree whose root has children according to $\chi$. To each child, mutually independently, we attach a random state in $\Sigma$ that follows the distribution $\vec{x}$. For every realization of this random procedure, we determine the state at the root using the rules of the automaton $A$. We then set $\Psi(\vec{x})$
to be the distribution of the random state thus induced at the root.

\begin{lemma}\label{iota distribution}
  Let $T\sim\GW(\chi)$.
If $\iota$ is an interpretation for the tree automaton $A$, then the distribution of $\iota(T)$ is a fixed point of the automaton distribution map $\Psi$.
\end{lemma}
\begin{proof}
Let the distribution of $\iota(T)$ be $\vec{y} = (y_{\sigma}: \sigma \in \Sigma)\in D$. Let $\tau\colon V(T) \rightarrow \Sigma$ be the assignment defined by $\tau(v) = \iota(T(v))$ for all $v \in V(T)$, which is almost surely compatible with $A$ by definition of interpretation. 

Under the labeling $\tau$, the 
state of the root $R_{T}$ is distributed as $\vec{y}$, since $\tau(R_T)=\iota(T)$. On the other hand, $R_{T}$ has children according to the distribution $\chi$; each of these children has an independent copy of $T$ descending from it. So, from the definition of $\tau$, the children of $R_{T}$ have i.i.d.\ labels distributed as $\vec{y}$. Hence the corresponding label at the root is $\Psi(\vec{y})$, by definition of $\Psi$. This shows that $\Psi(\vec{y}) = \vec{y}$, which is what we claimed.
\end{proof}

For all tree automata, the automaton distribution map $\Psi$ has at least one fixed point.
This holds because $\Psi$ is a continuous map from a finite-dimensional simplex to itself, and so
the Brouwer fixed-point theorem guarantees the existence of a fixed point.

For a given automaton $A$ and child distribution $\chi$,
suppose $\vec{\nu}$ is some fixed point of $\Psi$. We call $\iota$ an \emph{interpretation of the automaton $A$ corresponding to $\vec{\nu}$} if $\iota$ is indeed an interpretation of $A$ and $\iota(T) \sim \vec{\nu}$.
It is not hard to show that up to measure zero changes, there is at most one interpretation
corresponding to a given fixed point (see Proposition~\ref{prop:zero.one}).
If such an interpretation exists, we call $\vec{\nu}$ \emph{interpretable}; otherwise, we call it
\emph{rogue}. Our main results are a criterion for determining whether a given fixed point
is rogue or interpretable when $\abs{\Sigma}=2$ (Theorem~\ref{main 2 colours}), as well as
a sufficient condition for interpretability for $\abs{\Sigma}\geq 3$ (Theorem~\ref{thm:k.state}).
To state this criterion, we must define the two randomly
coloured trees explained in the next two sections.

\subsection{The random state tree}\label{subsec:RST}
Fix a child distribution $\chi$, automaton $A$, and a fixed point $\vec{\nu}$ of
the resulting automaton distributional map $\Psi$.
The \emph{random state tree associated with $\vec{\nu}$} is a coloured
Galton--Watson tree. We write it as $(T,\omega)$, where $\omega\colon V(T)\to \Sigma$
is a random colouring of the tree $T$. It is defined by the following properties:
\begin{enumerate}[(i)]
  \item $T\sim\GW(\chi)$;\label{RST1}
  \item for every $n$, the conditional distribution of $\big(\omega(v): v \in L_{n}\big)$ given
    $T|_{n}$ is i.i.d.~$\vec{\nu}$;\label{RST2}
  \item $\omega$ is almost surely compatible with $A$.\label{RST3}
\end{enumerate}
\begin{prop}\label{kolmogorov}
  These properties uniquely determine the distribution of $(T,\omega)$.
\end{prop}

Essentially, the random state tree is defined up to height~$n$ by generating the first $n$ levels of $T$,
colouring the leaves i.i.d.~$\vec\nu$, and then colouring the first $n-1$ levels of the
tree according to the automaton. 
The distributions of coloured trees generated by this procedure turn out to be consistent
for different values of $n$, which is a consequence of $\vec{\nu}$ being a fixed point of $\Psi$.
Kolmogorov's extension theorem then shows the existence of the distribution of the entire coloured tree.
This is shown in detail in the proof of Proposition~\ref{kolmogorov}, which we give in Section~\ref{sec:2}.

The colouring of the vertices of the random state tree is reminiscent of an interpretation,
which also yields a colouring of the tree via \eqref{interpretation}. 
But note that for a given fixed point of $\Psi$, the random state tree colouring always exists,
and it is a random colouring (on top of the randomness of the tree). On the other hand, given
a fixed point of $\Psi$, there may be no interpretations associated with it;
if there is an interpretation, the colouring it yields is deterministic given the tree.

\subsection{Definition of the pivot tree}\label{sec:pivot.def}

We now describe the pivot tree, leaving its formal definition to Section~\ref{pivot tree in subcritical}. 
Consider some vertex of $(T,\omega)$, and imagine changing its colour
and then recolouring all the vertices above it according to the rule of the automaton.
We call this recolouring operation a \emph{switching}.
If the switching changes the colour at the root, then we call the vertex \emph{pivotal for $(T,\omega)$}.
It is not hard to see that a vertex can only be pivotal if its parent is pivotal.
The subgraph of $T$ induced by the pivotal vertices is thus a subtree, which we call the
\emph{pivot tree} $\Tpiv$.
As we will see in Proposition~\ref{prop:Tpiv.GW}, the pivot tree is a multitype Galton--Watson
tree. 

We mention that the pivot tree is a bit more complicated when there are more than two states,
because a vertex can change colours in more than one way. However, to state Theorem~\ref{thm:k.state},
we need only use the \emph{pivot tree with maximal target set}, in which a vertex is pivotal if its colour
can be switched to any other colour with the result of changing the colour of the root in any way.


\subsection{The main result}\label{sec:main.results}
For all of our results, fix a child distribution $\chi$, an automaton $A$ on a finite set
of states $\Sigma$, and let $\Psi\colon D\to D$ be the automaton distributional map corresponding
to $A$ and $\chi$, defined in Section~\ref{fixed points}. 

First, as we mentioned, there is at most one interpretation for each fixed point:
\begin{prop}\label{prop:zero.one}
  If $\iota,\iota'\colon\Tt\to\Sigma$ are interpretations
  of $A$ under $\GW(\chi)$ corresponding to the same fixed point of the automaton
  distribution map, 
  then $\iota=\iota'$ a.e.-$\GW(\chi)$.
\end{prop}

Now, we give our main results.
Let $\vec\nu$ be a fixed point of $\Psi$.
We assume that the support of the probability distribution $\vec\nu$ is all of $\Sigma$; that is,
as a vector, all entries of $\vec\nu$ are nonzero. This is in fact no restriction,
since if $\vec\nu$ is supported on a subset of $\Sigma$, we can simply remove the extra elements
of $\Sigma$ and view $A$ as an automaton on this smaller set.
Recall that the pivot tree associated with $\vec\nu$ is a multitype Galton--Watson tree,
which will be proven in Proposition~\ref{prop:Tpiv.GW}.
We define a multitype Galton--Watson tree to be subcritical, critical, or supercritical
depending on whether its matrix of mean offspring sizes has spectral radius
smaller than, equal to, or greater than $1$ (see Section~\ref{sec:pivot.regularity}).
\begin{thm}\label{main 2 colours}
  Suppose that $\abs{\Sigma}=2$ and $\chi$ has finite logarithmic moment.
  Then $\vec{\nu}$ admits an interpretation if and only if the pivot tree associated
  with $\vec\nu$ is subcritical or critical.
\end{thm}
This theorem completely classifies fixed points as interpretable or rogue when $\abs{\Sigma}=2$.
It is practical to apply (see Section~\ref{sec:examples} for some examples), since
it only takes a computation to check the criticality of a given Galton--Watson tree.

When $\abs{\Sigma}\geq 3$, we give only a sufficient condition for existence of an interpretation.
\begin{thm}\label{thm:k.state}
  If the pivot tree with maximal target set associated with $\vec\nu$ is subcritical,
  then $\vec\nu$ admits an interpretation.
\end{thm}

Our full version of this result, Proposition~\ref{prop:subcritical.implies.interpretable}, 
is actually slightly stronger and applies in some cases when the pivot tree is critical
(see Remark~\ref{rmk:critical.case}).

\subsection{Connections to other work}\label{sec:connections}
This work has some concrete connections with mathematical logic.
We start by defining \emph{first-order} and \emph{monadic second-order} logic on trees.
A sentence in the first-order language for rooted trees is a finite combination of the following:
\begin{itemize}
  \item a constant symbol $R$ representing the root;
  \item a function $\pi$ where $\pi(v)$ represents the parent of vertex~$v$;
  \item a relation $=$, denoting equality of vertices;
  \item the Boolean connectives;
  \item existential and universal quantifications over vertices.
\end{itemize}
For example, a valid first order
sentence is that some vertex has exactly one child, which is expressed in the formal language
by
\begin{align*}
  \exists x\ \exists y\ \Bigl((\pi(y)=x) \land \bigl(\forall z\ (\pi(z)=x \implies z=y)\bigr)\Bigr).
\end{align*}
The monadic second-order language adds
\begin{itemize}
  \item existential and universal quantifications over sets of vertices;
  \item the relation $\in$, denoting set membership.
\end{itemize}
For example, the following sentence states that the tree is infinite:
\begin{align*}
  \exists S\ \forall x\ \Bigl((x\in S) \implies \bigl(\exists y\ (\pi(y)=x) \land (y\in S)\bigr) \Bigr).
\end{align*}
The \emph{quantifier depth} of a sentence in either language
is the maximal depth of nesting of existential and universal qualifiers.
In the example above, the quantifier depth is $3$.

Using Ehrenfeucht games, one can partition the set of rooted trees into finitely many
types by the relation that two trees have the same type if they
have the same truth value for all first-order sentences of quantifier depth at most $k$
(see \cite[Chapter~3]{Libkin}). Call this partition the \emph{rank-$k$ types}.
One can do the same replacing first-order logic with monadic second-order logic, producing
the \emph{MSO rank-$k$ types} \cite[Section~7.2]{Libkin}.
In both cases, one can deduce the type of a given tree $t$ 
from the types of the trees rooted at the children of the root of $t$.
This gives rise to tree automata on the set of rank-$k$ and MSO rank-$k$ types,
both of which have interpretations
given by mapping a tree to its type.

In \cite{PS1,PS2}, this automaton is investigated for the first-order case.
The most fundamental result of \cite{PS2} is that its automaton distribution map
is a contraction and hence has a unique fixed point. As a consequence,
since the
at-least-two property has multiple fixed points,
the property of a tree containing an infinite binary tree starting from its root cannot
be expressed in first-order logic.
We discuss this further in Section~\ref{sec:first.order}.
Our initial motivation for this paper was to make sense of the meaning of multiple fixed points.

  Our work also has some connections to the theory of \emph{recursive distributional equations} (RDEs)
  as developed by Aldous and Bandyopadhyay \cite{AB}.
  A prototypical example of an RDE is for the height of a Galton--Watson tree.
  Given a child distribution, let $N$ be the number of children of the root.
  Then the height of the tree $H$ satisfies the distributional equation
  \begin{align*}
    H &\overset{d}= 1 + \max(H_1,\ldots,H_N),
  \end{align*}
  where $H_1,H_2,\ldots$ are independent copies of $H$.
  
  For a given tree automaton, the automaton distribution map $\Psi$ defines an RDE.
  For any choice of fixed point $\vec\nu$,
  the random state tree $(T,\omega)$ is an example
  of an object introduced by Aldous and Bandyopadhyay called a 
  \emph{recursive tree process} (RTP). 
  RTPs are classified as \emph{endogenous}
  or \emph{nonendogenous}, which for $(T,\omega)$ corresponds to whether
  $\omega(R_T)$ is measurable with respect to $T$. In Proposition~\ref{prop:interpretability},
  we show that this is equivalent to interpretability of $\vec\nu$.
  Thus, Theorems~\ref{main 2 colours} and \ref{thm:k.state} can be viewed
  as criteria for the endogeny of an RTP, for RTPs in a certain class.
  This work or extensions of it might prove useful, as the endogeny
  of RTPs is an actively pursued topic (see \cite{Endog1,Endog2,Endog3,Endog4}, for example).
  
  Two very recent papers have a similar flavour as ours. In \cite{BDF}, the authors
  consider critical Galton--Watson trees conditioned to have $n$ vertices.
  Each vertex of the tree is given a label from a finite set. The label of a parent
  is a function of the labels of the children along with an independent set of randomness.
  (This is also the case with Aldous and Bandyopadhyay's definition of a recursive tree process.)
  The main result of the paper is a limit theorem for the distribution of the label of the root
  as $n\to\infty$.
  
  The paper \cite{MS} considers Galton--Watson trees labeled by elements of $[0,1]$, cut off
  at level~$2n$. Leaves are assigned independent labels sampled uniformly from $[0,1]$.
  Then, the label at a parent at an even generation is the minimum of its children's labels; at
  an odd generation, it is the maximum of its children's labels. This models a game in which
  two players take turns, one trying to make the score big and one trying to keep it small.
  The paper classifies possible limit distributions for the label at the root as $n\to\infty$.
  It also investigates endogeny, the question of whether the value at the root is determined
  by the structure of the tree.
  
\subsection{Outline}

In Section~\ref{sec:2}, we first establish basic properties
of interpretations, fixed points, the random state tree, and the pivot tree used
throughout this paper.
In Section~\ref{sec:subcritical}, we prove the first direction of Theorem~\ref{main 2 colours}, 
existence of an interpretation when the pivot tree is almost surely finite.
The main tool for this is the Kahn--Kalai--Linial
inequality from the theory of Boolean functions \cite{KKL}. The other direction of 
Theorem~\ref{main 2 colours} is proven in Section~\ref{sec:supercritical} using the spine 
decomposition technique pioneered by Lyons, Peres, and Pemantle \cite{LPP}. Finally,
in Section~\ref{sec:examples}, we apply these results to answer Question~\ref{q:spencer}.
We also give examples exhibiting a phase transition between interpretable and rogue
for a fixed point as the child distribution of the tree is varied.
In Section~\ref{sec:open}, we discuss some open questions.

\section{Foundational properties of our objects}\label{sec:2}

In this section, we fix a child distribution $\chi$, an automaton $A$ on a set of states $\Sigma$, and a fixed point $\vec{\nu}$ of the automaton distributional map $\Psi\colon D\to D$ determined by $A$ and $\chi$.
We will demonstrate some of the basic properties of
fixed points, interpretations, the random state tree, and the pivot tree.

\subsection{The random state tree}
We now give the proof of Proposition~\ref{kolmogorov}, establishing the existence
of the random state tree $(T,\omega)$ defined in Section~\ref{subsec:RST}.
We then show in Proposition~\ref{(T, omega) multi-type} that it is a multitype Galton--Watson tree.

\begin{proof}[Proof of Proposition~\ref{kolmogorov}]
To invoke the Kolmogorov extension theorem \cite[Theorem~6.16]{Kallenberg}, we must construct a sequence of random variables $(T_n,\omega_n)$ such that $T_n$ is the truncation
  to level~$n$ of a $\GW(\chi)$-distributed tree, the distribution of $(\omega_n(v))_{v\in L_n(T_n)}$
  conditional on $T_n$
  is i.i.d.-$\vec\nu$, the values of $\omega_n(v)$ for $v$ in levels $0,\ldots,n-1$ are as given
  by the automaton, and the truncation of $(T_{n+1},\omega_{n+1})$ to $n$ levels is distributed
  as $(T_n,\omega_n)$. (Formally speaking, to apply the Kolmogorov extension theorem, we view
  labeled trees as a sequence of their finite truncations, but we will ignore these details.)
  
  To construct $(T_n,\omega_n)$, we simply define $T_n$ as the truncation of a Galton--Watson tree, then
  colour the level~$n$ vertices i.i.d.-$\vec\nu$, and then colour levels~$0,\ldots,n-1$ of the tree
  according to the automaton. The crux of the proof is showing that the truncation of $(T_{n+1},\omega_{n+1})$
  to level~$n$ is distributed as $(T_n,\omega_n)$. Clearly, $T_{n+1}|_n$ is distributed as $T_n$, and
  the colouring given by $\omega_{n+1}$ on levels~$0,\ldots,n-1$ of $T_{n+1}|_n$ is as induced by
  the automaton. We need only show that conditional on $T_{n+1}|_n$, the labeling $\omega_{n+1}$
  assigns i.i.d.\ $\vec\nu$ colours to the level~$n$ vertices.
  
  To see this, recall how we define $\Psi(\vec{x})$: We let a node have children according to distribution $\chi$. Each of these children is assigned, mutually independently, a state according to distribution $\vec{x}$.
  The induced random state of the parent node, obtained via the rules of $A$, has distribution $\Psi(\vec{x})$.
  Meanwhile, each $v\in L_n(T_{n+1})$ has children according to $\chi$, these children
  receive i.i.d.-$\vec\nu$ labels from $\omega_{n+1}$, and $\omega_{n+1}(v)$ is given by applying
  the automaton to these labels. Hence, 
  the distribution of $\omega_{n+1}(v)$ conditional on $T_{n+1}|_n$ is
  $\Psi(\vec\nu)$. As $\vec\nu$ is assumed to be a fixed point, this equals $\vec\nu$.
  The values of $\omega_{n+1}(v)$ are independent for the different level~$n$ vertices $v$
  conditional on $T_{n+1}|_v$, showing that $\omega_{n+1}$ assigns i.i.d.\ $\vec\nu$ colours
  to the level~$n$ vertices.
\end{proof}

Now that we have shown the existence of the random state tree, we prove that it is Galton--Watson
with types given by $\omega$.

\begin{prop}\label{(T, omega) multi-type}
The random state tree $(T, \omega)$ is a multitype Galton--Watson tree.
\end{prop}
\begin{proof}
  For $\sigma_1,\ldots,\sigma_k\in\Sigma$, let 
  \begin{align*}
    \chicol(\sigma_1,\ldots,\sigma_k) = \chi(k)\vec\nu(\sigma_1)\cdots\vec\nu(\sigma_k),
  \end{align*}
  the probability that $R_T$ has exactly $k$ children and that their types in order
  are $\sigma_1,\ldots,\sigma_k$.
  Let $\chicol^{\sigma}(\sigma_1,\ldots,\sigma_k)$ denote the conditional probability that $R_T$ has
  exactly $k$ children and that their types in order are $\sigma_1,\ldots,\sigma_k$,
  given that $\omega(R_T)=\sigma$. Thus, if $\sigma$ is the type according to $A$
  for a vertex with children of types $\sigma_1,\ldots,\sigma_k$, then
  \begin{align}\label{eq:chi.condition}
    \chicol(\sigma_1,\ldots,\sigma_k) = \vec\nu(\sigma)\chicol^\sigma(\sigma_1,\ldots,\sigma_k).
  \end{align}
  
  Our goal is to prove that conditional on the first $n$ levels of $(T,\omega)$, each
  vertex~$v$ at level~$n$ independently gives birth according to the distribution given by
  $\chicol^{\omega(v)}$.
  Fix any $(t,\tau)\in\Tcol$.
  By definition of $\omega$,
  \begin{align*}
    \P\Bigl[ (T,\omega)\in[t,\tau]_{n+1} \Bigr] = \P\Bigl[T\in[t]_{n+1}\Bigr] 
                                                 \prod_{u\in L_{n+1}(t)}\vec\nu\bigl(\tau(u)\bigr),
  \end{align*}
  recalling the notation $[t]_n$ and $[t,\tau]_n$ defined in Section~\ref{notations}.
  For a vertex $v\in V(t)$, let $C(v)$ denote its children in $t$.
  Since $T$ is Galton--Watson with child distribution $\chi$,
  \begin{align*}
    \P\Bigl[ (T,\omega)\in[t,\tau]_{n+1} \Bigr]
      &=\Biggl(\P\Bigl[T\in[t]_n\Bigr]\prod_{v\in L_n(t)}\chi\bigl(\abs{C(v)}\bigr)\Biggr)\prod_{u\in L_{n+1}(t)}\vec\nu\bigl(\tau(u)\bigr)\\
      &= \P\Bigl[T\in[t]_n\Bigr]\prod_{v\in L_n(t)}\Biggl(\chi\bigl(\abs{C(v)}\bigr)\prod_{u\in C(v)}\vec\nu\bigl(\tau(u)\bigr)\Biggr)\\
      &= \P\Bigl[T\in[t]_n\Bigr]\prod_{v\in L_n(t)}\chicol\bigl(\tau(u)_{u\in C(v)}\bigr).
  \end{align*}
  By \eqref{eq:chi.condition}, this becomes
  \begin{align*}
    \P\Bigl[ (T,\omega)\in[t,\tau]_{n+1} \Bigr]
      &= \P\Bigl[T\in[t]_n\Bigr]\prod_{v\in L_n(t)}\vec\nu(\tau(v))\chicol^{\tau(v)}\bigl(\tau(u)_{u\in C(v)}\bigr)\\
      &= \P\Bigl[(T,\omega)\in[t,\tau]_n\Bigr]\prod_{v\in L_n(t)}\chicol^{\tau(v)}\bigl(\tau(u)_{u\in C(v)}\bigr),
  \end{align*}
  which is exactly what we set out to prove.  
\end{proof}

\subsection{Equivalent conditions for interpretability of fixed points} 
\label{sec:equivalent.conditions}
We start with a definition that will come up again elsewhere in the paper.
Given a rooted tree $t$ and a colouring of its level~$n$ vertices, we can repeatedly
apply the automaton $A$ to determine the state of the root. We define $A^n_t\colon\Sigma^{\ell_n(t)}\to\Sigma$
to be the result of doing so, considering it as a map from the colours at level~$n$ to a colour at the root.

Now, we show that a given fixed point can have at most one interpretation:
\begin{proof}[Proof of Proposition~\ref{prop:zero.one}]
  Viewing the statement of the proposition probabilistically, our goal
  is to show that $\iota(T)=\iota'(T)$ a.s.
Fix some $\sigma \in \Sigma$.
We first show that for any $n$,
\begin{align}\label{eq:iota=iota'}
  \P\bigl[\iota(T) = \sigma\bigmid T|_{n}\bigr] = \P\bigl[\iota'(T) = \sigma\bigmid T|_{n}\bigr]\text{ a.s.}
\end{align}
To prove this, we start by observing that $\iota(T)$ is determined by 
  $\bigl(\iota(T(v))\bigr)_{v\in L_n}$.
Indeed, since $\iota$ is an interpretation of $A$ and thus respects the automaton,
\begin{align*}
  \iota(T) = A^n_T\Bigl(\bigl(\iota(T(v))\bigr)_{v\in L_n}\Bigr).
\end{align*}
Conditional on $T|_n$, each tree $T(v)$ for $v\in L_n$ is independent and distributed
identically to $T$.
Let $\vec\nu$ be the fixed point corresponding to $\iota$ and $\iota'$.
Since the distribution of $\iota(T)$
is $\vec\nu$, the distribution of
$(\iota(T(v)))_{v\in L_n}$ conditional on $T|_n$ is i.i.d.\ $\vec\nu$.
Therefore,
\begin{align}\label{eq:iota.cond.awk}
  \P\bigl[\iota(T) = \sigma\bigmid T|_{n}\bigr] = \P\Bigl[A^n_T\Bigl(\bigl(\omega(v)\bigr)_{v\in L_n}\Bigr)=\sigma\Bigmid T|_n\Bigr]
  \text{ a.s.,}
\end{align}
recalling that by its definition, 
the colouring $\omega$ of the random state tree $(T,\omega)$ also assigns colours
to the level~$n$ vertices by sampling independently from $\vec\nu$, conditional on $T|_n$.
The exact same reasoning shows that
\begin{align*}
  \P\bigl[\iota'(T) = \sigma\bigmid T|_{n}\bigr] = \P\Bigl[A^n_T\Bigl(\bigl(\omega(v)\bigr)_{v\in L_n}\Bigr)=\sigma\Bigmid T|_n\Bigr]
  \text{ a.s.,}
\end{align*}
which proves \eqref{eq:iota=iota'}.

Now, we take limits as $n\to\infty$ to complete the proof.
The $\sigma$-fields generated by $T|_{n}$ form a filtration that converges to the $\sigma$-field generated by $T$. Hence, by L\'{e}vy's upward theorem, 
\begin{align}\label{levy upward 1}
\P\big[\iota(T) = \sigma\bigmid T|_{n}\big] &\rightarrow \P\big[\iota(T) = \sigma\bigmid T\big] = \mathbf{1}\{\iota(T) = \sigma\} \text{ a.s.}\\\intertext{and} 
\label{levy upward 2}
\P\big[\iota'(T) = \sigma\bigmid T|_{n}\big] &\rightarrow \P\big[\iota'(T) = \sigma\bigmid T\big] = \mathbf{1}\{\iota'(T) = \sigma\} \text{ a.s.}
\end{align}
By \eqref{eq:iota=iota'}, these two limits are identical.
We conclude that $\mathbf{1}\{\iota(T) = \sigma\} = \mathbf{1}\{\iota'(T) = \sigma\}$ a.s.\ for all $\sigma\in\Sigma$.
\end{proof}
The expression $A^n_T\bigl((\omega(v))_{v\in L_n}\bigr)$ in \eqref{eq:iota.cond.awk}
is equal to $\omega(R_T)$, since the colouring $\omega$
is compatible with $A$. Thus \eqref{eq:iota.cond.awk} can be written as
\begin{align}\label{eq:iota.cond}
    \P\bigl[\iota(T) = \sigma\bigmid T|_{n}\bigr] = \P\Bigl[\omega(R_T)=\sigma\Bigmid T|_n\Bigr]
  \text{ a.s.,}
\end{align}
which will come up again in the next proposition.
Before we state it,
we mention a standard characterization of measurability 
\cite[Lemma~1.13]{Kallenberg}:
Let $X$ and $Y$ be random variables taking values in measurable spaces $\mathcal{X}$ and
$\mathcal{Y}$, respectively, with $\mathcal{X}$ assumed to be a Polish space
endowed with its Borel $\sigma$-algebra. Then
the measurability of $X$ with respect to $Y$ is equivalent
to existence of a measurable map $f\colon\mathcal{Y}\to\mathcal{X}$ such that
$X=f(Y)$ a.s.
\begin{prop}\label{prop:interpretability}
  The following statements are equivalent:
  \begin{enumerate}[(i)]
    \item $\vec\nu$ is interpretable;  \label{i:interpretable}
    \item for each $\sigma\in\Sigma$, \label{i:limit.interpretable}
      \begin{align*}
        \lim_{n\to\infty} \P\bigl[\omega(R_T)=\sigma\bigmid T|_n\bigr]\in \{0,1\}\text{ a.s.;}
      \end{align*}
    \item $\omega(R_T)$ is measurable with respect to $T$; \label{i:root.measurable}
    \item $\omega$ is measurable with respect to $T$.  \label{i:omega.measurable}
  \end{enumerate}
\end{prop}
\begin{proof}[Proof that \ref{i:interpretable} $\implies$ \ref{i:limit.interpretable}]
Let $\iota$ be the interpretation of automaton $A$ corresponding to $\vec{\nu}$ (it is unique up to $\GW(\chi)$-negligible sets by Proposition~\ref{prop:zero.one}). 
By \eqref{eq:iota.cond}, 
\begin{align*}
   \lim_{n\to\infty}\P\bigl[\omega(R_{T}) = \sigma\bigmid T|_{n}\bigr]
    &= \lim_{n\to\infty}\P\bigl[\iota(T) = \sigma\bigmid T|_{n}\bigr]
    = \mathbf{1}\{\iota(T) = \sigma\} \in \{0, 1\} \text{ a.s.,}
\end{align*}
applying L\'evy's upward theorem as in \eqref{levy upward 1}.
\end{proof}
\begin{proof}[Proof that \ref{i:limit.interpretable} $\implies$ \ref{i:root.measurable}]
   Invoking L\'{e}vy's upward theorem and then \ref{i:limit.interpretable},
\begin{equation}\label{levy upward 3}
\P\bigl[\omega(R_T)=\sigma\bigmid T\bigr] =\lim_{n\to\infty}\P\bigl[\omega(R_T)=\sigma\bigmid T|_n\bigr] \in \{0,1\}
  \text{ a.s.}
\end{equation}
Thus, given the entire tree $T$, we can almost surely determine whether $\omega(R_{T})$ equals $\sigma$ or not. Since this is true for every $\sigma \in \Sigma$, the state $\omega(R_T)$ is almost surely
equal to a deterministic function of $T$, showing that $\omega(R_{T})$ is measurable with respect to $T$.
\end{proof}
\begin{proof}[Proof that \ref{i:root.measurable} $\implies$ \ref{i:omega.measurable}]
   Fix any $v \in L_{n}$. Let $\omega|_{T(v)}$ denote the restriction of $\omega$ on the subtree $T(v)$. 
   The conditional distribution of the coloured tree $\big(T(v), \omega|_{T(v)}\big)$ 
   given $T|_{n}$ is the same as the unconditional distribution of  $(T, \omega)$. By \ref{i:root.measurable}, we know that $\omega(v)$ is measurable with respect to $T(v)$ and is hence an almost sure
   function of $T(v)$. As $T$ has countably many vertices, we can write $\omega$ as an almost sure function of $T$.
\end{proof}
\begin{proof}[Proof that \ref{i:omega.measurable} $\implies$ \ref{i:interpretable}]
   Since $\omega$ is measurable with respect to $T$, so is $\omega(R_T)$. Therefore there
   exists a measurable map $\iota\colon \mathcal{T}\to\Sigma$ such that $\iota(T)=\omega(R_T)$ a.s. 
   We claim that this will serve as the desired interpretation:
    Since $\omega$ is almost surely compatible with $A$, the assignment $v\mapsto\iota(T(v))$ is also almost surely compatible with $A$ and is hence an interpretation. Furthermore, from the construction of $\omega$, we know that $\omega(R_{T})$ will be distributed as $\vec{\nu}$, and hence so is $\iota(T)$.
\end{proof}

We mentioned at the end of Section~\ref{subsec:RST} that the colouring of $T$ given by $\omega$ and
the colouring given by an interpretation via \eqref{interpretation} are in general different.
However, it is a consequence of Proposition~\ref{prop:interpretability} that
when an interpretation exists for a given
fixed point, the two colourings are the same:
\begin{cor}\label{prop:RST.criterion}
  The fixed point $\vec\nu$ is interpretable if and only if $\omega$ is measurable with respect
  to $T$. If this occurs, then $\omega(R_t)$ is determined by $t$ for $\GW(\chi)$-a.e.
  $t\in\Tt$, and the resulting map $\Tt\to\Sigma$ given by $t\mapsto\omega(R_t)$
  is the the unique interpretation corresponding to the fixed point, up to
  a.e.-$\GW(\chi)$ equivalence.
\end{cor}
\begin{proof}
  The equivalence of interpretability and measurability of $\omega$ with respect to $T$
  is one part of Proposition~\ref{prop:interpretability}. In the proof
  that \ref{i:omega.measurable} implies \ref{i:interpretable}, it is shown
  that $t\mapsto\omega(R_t)$ yields an interpretation corresponding to the given fixed point.
  The uniqueness of this interpretation is given by Proposition~\ref{prop:zero.one}.
\end{proof}

The equivalences proven in this section reduce the question of whether a fixed point of $\Psi$
is rogue or interpretable to whether the colouring $\omega$ in the random state tree
$(T,\omega)$ is random or deterministic given $T$. This question is on its face no easier than the
original one. To answer it, the key will be the pivot tree,
a random subtree of $(T,\omega)$ that we discuss now.

\subsection{The pivot tree}\label{pivot tree in subcritical}

We start with some notation.
Suppose we are given a coloured tree $(t,\tau)$ with $\tau$ compatible with $A$. 
Suppose $v\in L_n(t)$. Now, imagine that we change the colour of $v$ to some $\gamma \in \Sigma \setminus \{\tau(v)\}$, and then recolour the vertices at levels~$0,\ldots,n-1$ based on this. We say that we
have \emph{switched the colour at $v$ to $\gamma$}, and we denote the new colouring by
$\tau^{v\to\gamma}$. Note that $\tau^{v\to\gamma}$ is only defined on $t|_n$, and that it is consistent
with the automaton at levels~$0,\ldots,n-1$.

Now, we give the full definition of the pivot tree.
When $\abs{\Sigma}=2$, this definition is simple: the pivot tree of $(t,\tau)$
consists of the subgraph induced by all vertices $v$ such that switching $\tau$ at $v$ 
changes the value of the root.
We denote the pivot tree of $(T,\omega)$ by $\Tpiv$, which 
we will prove shortly is indeed a tree.
See Figure~\ref{fig:pivotal.two.state} for an example.

When $\abs{\Sigma}\geq 3$, we sometimes demand that the colour of the root change to one of
a specific set of colours, known as the \emph{target set},
complicating the definition. Given $(t,\tau)$ with $\tau$ compatible with $A$, let
$\Aa\subseteq\Sigma\setminus\{\tau(R_t)\}$ represent this target set.
Given $t$, $\tau$, and $\Aa$, for any $v\in V(t)$ we define
\begin{align*}
  B_v=\bigl\{\gamma\in\Sigma\colon \tau^{v\to\gamma}(R_t)\in\Aa \bigr\}.
\end{align*}
In other words, $B_{v}$ is the set of colours such that switching $v$ to an element of $B_v$
changes the colour of the root to an element of $\Aa$. For any $v\in V(t)$, we say that
$v$ is \emph{pivotal for $(t,\tau)$ with target set $\Aa$} if $B_v\neq\emptyset$.

To define the pivot tree of $(T,\omega)$, we must specify a target set for each possible state of
the root. For each $\sigma\in\Sigma$, 
let $\emptyset\neq\Aa_\sigma\subseteq \Sigma\setminus\{\sigma\}$ be a given (deterministic) set 
that we call the \emph{target set of the root at state~$\sigma$}. The most basic example is to set
$\Aa_\sigma=\Sigma\setminus\{\sigma\}$
for all $\sigma$, which corresponds to requiring the colour of the root to change without caring
what it changes to.
Let $\Aa=(\Aa_\sigma)_{\sigma\in\Sigma}$. We define the pivot tree,
$\Tpiv=\Tpiv(\Aa)$, as the subgraph of $T$ induced by all vertices
pivotal for $(T,\omega)$ with target set $\Aa_{\omega(R_T)}$.
The pivot tree is measurable with 
respect to $(T,\omega)$; that is, $\Tpiv$ is a measurable function of $(T,\omega)$.
Also, observe that this definition works in the $\abs{\Sigma}=2$ case as well.
Here, there is only one possible choice of $\Aa_\sigma$, and either $B_v=\emptyset$ or $B_v$ is a singleton set
made up of the opposite colour as $\omega(v)$.

\begin{prop}\label{prop:Tpiv.GW}
  For given target sets $(\Aa_\sigma)_{\sigma\in\Sigma}$, assign the
  type $(\omega(v); B_v)$ to each vertex $v\in V(T)$. 
  With these types, both $T$ and $\Tpiv$ are multitype Galton--Watson trees.
\end{prop}
\begin{proof}
  We start with proof for $T$.
  Let $\Fff_n$ denote the $\sigma$-algebra generated by $T|_n$ and by the types
  $(\omega(v); B_v)$ for vertices $v$ up to level~$n$. We will refer to these
  as augmented types, in contrast with the unaugmented types given by $\omega$ alone.
  
  We must show that conditional on $\Fff_n$, the vertices
  at level~$n$ independently give birth according to their augmented types.
  First, we observe that the values of $B_v$ for $v$ in $T|_n$
  are determined by the first $n$ levels of $(T,\omega)$. Hence, 
  conditioning on $\Fff_n$ is the same as conditioning
  on the first $n$ levels of $(T,\omega)$. Thus, by Proposition~\ref{(T, omega) multi-type},
  conditional on $\Fff_n$, each vertex $v$ at level~$n$ independently gives birth
  to children whose number and unaugmented type are determined by the unaugmented type of $v$.
  
  Now, we just need to extend this statement to the augmented types.
  The key fact is the following: Let $u_1,\ldots,u_k$ be the children of some node $v$.
  Then for each $i=1,\ldots,k$, the set $B_{u_i}$ is determined by $\omega(u_1),\ldots,\omega(u_k)$
  and $B_v$.
  Indeed, from $\omega(u_1),\ldots,\omega(u_k)$, we can determine the effect on the colour
  of $v$ of changing $u_i$ to have any given colour. From $B_v$, we know whether the change
  will alter the colour of the root to have a value in $\Aa_{\omega(R_T)}$. Thus we can
  determine $B_{u_i}$.
  
  Let $C(v)$ denote the children of a vertex $v$, as in Proposition~\ref{(T, omega) multi-type}.
  From the fact above, conditional on $\Fff_n$, the distribution of $\bigl(B_u:u\in C(v)\bigr)$
  for any $v\in L_n$ is determined by $(\omega(v); B_v)$. This completes the proof
  that $T$ is multitype Galton--Watson with the augmented types.
  
  To prove the statement for $\Tpiv$, we first observe that $\Tpiv$ is indeed a tree,
  since if a vertex $u$ has $B_u\neq\emptyset$, then its parent $v$ evidently satisfies
  $B_v\neq\emptyset$. Thus, $\Tpiv$ is the tree formed by ignoring vertices of certain types
  in the Galton--Watson tree $T$, which always creates another Galton--Watson tree.
\end{proof}

\begin{figure}
    \begin{tikzpicture}[vert/.style={circle,draw=black,inner sep=2pt},
                        bvert/.style={circle,draw=black,inner sep=2pt,line width=.1cm}, 
                        yscale=1.15,xscale=1.6]
    \path (0,0) node[bvert,label=right:{\small $\{0,1\}$}] (R) {$2$};
    \path (R) + (-1.5,-1) node[vert,label=right:{\small $\emptyset$}] (R1) {$0$}    
                              + (0,-1) node[bvert,label=right:{\small $\{0\}$}] (R2) {$1$};
    \path (R)+ (1.5,-1) node[bvert,label=right:{\small $\{0\}$}] (R3) {$1$};
    \path (R2) + (0,-1) node[bvert,label=right:{\small $\{0\}$}] (R21) {$1$}
          (R3) + (-.5,-1) node[vert,label=right:{\small $\emptyset$}] (R31) {$0$}
          (R3) + (.5,-1) node[bvert,label=right:{\small $\{0\}$}] (R32) {$1$};
    \path[draw] 
        (R)--(R1) (R)--(R2) (R)--(R3)
        (R2)--(R21) (R3)--(R31) (R3)--(R32);    
    \end{tikzpicture}
  \caption{The automaton in this example is on $\{0,1,2\}$. The state of a parent is given
    by the sum of its children's states, capped at $2$. Bold vertices are pivotal with target set
    $\{0,1\}$. Written to the right of each vertex~$v$ is the set $B_v$, indicating which states
    $v$ can be switched to with the effect of changing the state of the root to a value in the target set.}
  \label{fig:pivot_tree}
\end{figure}
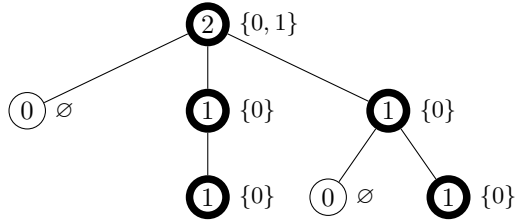

See Figure~\ref{fig:pivot_tree} for an example of a pivot tree when $\abs{\Sigma}\geq 3$.
In general, when we refer to $\Tpiv$ as a Galton--Watson tree from now on, we mean with 
types given as in Proposition~\ref{prop:Tpiv.GW}.
When $\abs{\Sigma}=2$, since either $B_v=\emptyset$ or $B_v$ is a singleton set for each $v$,
we can think of the type $(\omega(v);B_v)$
as simply $\omega(v)$ along with an indicator on $v$ being pivotal. Thus $\Tpiv$ in this
case is Galton--Watson with the types given by $\omega$ alone.
Naively, one might think that $(\Tpiv,\omega)$ would be Galton--Watson even when
$\abs{\Sigma}\geq 3$. We can see the problem with this in Figure~\ref{fig:pivot_tree}.
Let $u$ be the $0$-labeled vertex on the bottom level of the tree, and let $v$ be its parent.
Vertex~$u$ is not pivotal for the given target
set of the root (or indeed, for any possible target set). However, for
the subtree rooted at $v$, vertex~$u$ is pivotal for the target set $\{0,2\}$.
Thus, if we do not include the sets $B_v$ in the information given by the types, the law of
the progeny of a vertex would depend not just on the type of the vertex but on its ancestors.

\subsection{Regularity properties of the pivot tree}\label{sec:pivot.regularity}

For a given multitype Galton--Watson tree, define a matrix by setting
$M_{ij}$ to the expected number of offspring
of type~$j$ for a parent of type~$i$.
We classify the process as
subcritical, critical, or supercritical depending on whether the spectral radius of $M$
is smaller than, equal to, or greater than $1$.
If $M^n$ has strictly positive entries for some
choice of $n$, then the Galton--Watson process is called \emph{positive regular}. 
This says that it is possible for any type to have a
descendant of any other type, and that no periodic behaviour occurs.
The process is called \emph{singular} if each type gives birth to exactly one child with probability one.
Multitype Galton--Watson trees are nearly always considered under the assumption that they
are positive regular and nonsingular. Under this assumption, the process dies out with probability one
in the subcritical and critical cases, and it survives with positive probability in the supercritical
case. Regardless of the starting type, the expected size of the $n$th generation
vanishes exponentially in the subcritical case; remains of constant order in the critical case;
and grows exponentially in the supercritical case.

For a Galton--Watson tree without these assumptions, the situation is messier.
To illustrate, consider a process with two types $A$ and $B$ and matrix of means 
$M=\bigl[\begin{smallmatrix} 1&a\\0&1\end{smallmatrix}\bigr]$ for $a>0$. The expected
number of vertices of each type at level~$n$ starting with a vertex of type~$A$ is
given by the first row of $M^n$, which is $(1,an)$. Thus, even though this process
is critical, the expected size of the $n$th generation grows to infinity, though
only at a polynomial rate.
On the other hand, this tree still dies out with probability one,
as we can see by viewing it as a backbone
of a critical single-type Galton--Watson tree of vertices of type~$A$, each of which gives birth to 
critical single-type trees of vertices of type~$B$, all of which die out with probability one.

In general, without the assumption of positive regularity and nonsingularity, 
it is still correct that a subcritical tree has exponentially
vanishing expected $n$th generation and hence dies out almost surely.
By \cite{Sevastyanov} (see \cite[Theorem~10.1]{Harris}), so long as there does not
exist a collection of types $\Cc$ such that the children of a vertex of type in $\Cc$
include exactly one of the types in $\Cc$ with probability one, a critical tree dies out
almost surely; and a supercritical tree survives with positive probability from some starting
state.

The Galton--Watson tree $\Tpiv$ need not be positive regular.
Nonetheless, when $\abs{\Sigma}=2$, many features of positive regularity still hold.
We give a lemma that we will use to prove this.

\begin{lemma}\label{lem:growth.rates}
  Suppose that $\Sigma=\{0,1\}$.
  Let $Z_0$ and $Z_1$ be the number of children of $R_T$ pivotal for $(T,\omega)$ of types
  $0$ and $1$, respectively. Then
  \begin{align*}
    \E[ Z_0 \mid \omega(R_T)=0 ] &= \E[ Z_1 \mid \omega(R_T)=1 ],\\
    \intertext{and}
    \E[ Z_0 \mid \omega(R_T)=1 ] &= 
      \frac{\vec\nu(0)^2}{\vec\nu(1)^2}  \E[ Z_1\mid\omega(R_T)=0 ].
  \end{align*}
  Hence, if $M$ is the matrix of means of $\Tpiv$, given by $M=(m_{ij})_{i,j\in\{0,1\}}$ where
  $m_{ij} = \E[Z_j\mid\omega(R_T)=i]$, then
  \begin{align}\label{eq:M}
    M &= \begin{pmatrix} m_{00} & m_{01}\\ 
         \frac{\vec\nu(0)^2}{\vec\nu(1)^2}m_{01} & m_{00}
    \end{pmatrix}.
  \end{align}

\end{lemma}
\begin{proof}
  Given a list $\sigma=(\sigma_1,\ldots,\sigma_k)\in\{0,1\}^k$ representing
  the states of an ordered set of children, we abuse notation slightly and write
  $A(\sigma)$ to mean the value that the automaton assigns to the parent given these
  children. For example, if $\sigma=(0,0,1,0,1)$, then we write $A(\sigma)$ to denote
  $A(3,2)$, the type of the parent when there are three children of type~$0$ and two of
  type~$1$. We say that coordinate $\sigma_i$ is pivotal if switching its value changes
  $A(\sigma)$. For example, if $A$ is the at-least-two automaton of 
  Example~\ref{ex:at.least.two} and $\sigma$ is as above, then $\sigma_3$ and $\sigma_5$ 
  are pivotal.
  
  For $a,b\in\{0,1\}$, let
  \begin{align*}
    \begin{split}
    S_k(a,b) = \Bigl\{ (\sigma,i)\colon \sigma\in\{0,1\}^k,\; i\in\{&1,\ldots,k\},\;A(\sigma)=a,\;\\
                      &\sigma_i=b,\;
                      \text{and $\sigma_i$ is pivotal} \Bigr\},
    \end{split}
  \end{align*}
  representing a configuration of $k$ children making the parent have type~$a$
  and a choice of a pivotal child of type~$b$.
  There is a natural bijection between $S_k(a,b)$ and $S_k(1-a,1-b)$. The map is given
  by sending $(\sigma,i)\in S_k(a,b)$ to $(\sigma',i)\in S_k(1-a,1-b)$, where $\sigma'$
  is equal to $\sigma$ except at coordinate~$i$. Applying this bijection, keeping
  in mind that the states of the level~$1$ vertices of $(T,\omega)$ conditional on $T|_1$
  are i.i.d.\ $\vec\nu$,
  \begin{align*}
    \E\bigl[ Z_0\1\{\omega(R_T)=a\} \mid \text{$R_T$ has $k$ children} \bigr]
      &= \sum_{(\sigma,i)\in S_k(a,0)} \vec{\nu}^{\otimes n}(\sigma)\\
      &= \sum_{(\sigma,i)\in S_k(1-a,1)} \frac{\vec{\nu}(0)}{\vec{\nu}(1)}\vec{\nu}^{\otimes n}(\sigma)\\
      &= \frac{\vec{\nu}(0)}{\vec{\nu}(1)}
        \E\bigl[ Z_1\1\{\omega(R_T)=1-a\} \mid \text{$R_T$ has $k$ children} \bigr].
  \end{align*}
  Here we use the notation $\vec{\nu}^{\otimes n}$ to denote the $n$-fold product measure
  of $\vec\nu$ with itself.
  Taking expectations, in the $a=0$ case this yields
  \begin{align*}
    \frac{\E\bigl[ Z_0\1\{\omega(R_T)=0\}\bigr]}{\vec\nu(0)}
      &= \frac{\E\bigl[ Z_1\1\{\omega(R_T)=1\} \bigr]}{\vec\nu(1)},
  \end{align*}
  while in the $a=1$ case it yields
  \begin{align*}
    \frac{\E\bigl[ Z_0\1\{\omega(R_T)=1\}\bigr]}{\vec\nu(1)}
      &= \biggl(\frac{\vec\nu(0)^2}{\vec\nu(1)^2}\biggr)
        \frac{\E\bigl[ Z_1\1\{\omega(R_T)=0\} \bigr]}{\vec\nu(0)}.\qedhere
  \end{align*}
\end{proof}


This lets us prove that when $\abs{\Sigma}=2$, the pivot tree behaves nicely.
In particular, at criticality $\Tpiv$ dies out and has expected size one at every generation.

\begin{prop}\label{prop:2state.regularity}
  Suppose that $\Sigma=\{0,1\}$ and that both entries of $\vec\nu$ are positive.
  Let $M=(m_{ij})_{i,j\in\{0,1\}}$ be the matrix of means of $\Tpiv$.
  \begin{enumerate}[(a)]
    \item The largest eigenvalue of $M$ in absolute value is equal to $\E[ \ell_1(\Tpiv)]$.
      \label{i:eig}
    \item For all $n$, it holds that $\E[\ell_n(\Tpiv)] = \E[\ell_1(\Tpiv) ]^n$.
      \label{i:gen.size}
    \item If $\Tpiv$ is supercritical, then it is infinite with positive probability
      conditional on both $\omega(R_T)=0$ and on $\omega(R_T)=1$.
      \label{i:supercrit}
    \item If $\Tpiv$ is critical, then it is finite with probability one.
      \label{i:crit}
  \end{enumerate}
\end{prop}
\begin{proof}[Proof of \ref{i:eig}]
  By \eqref{eq:M} from Lemma~\ref{lem:growth.rates}, the characteristic polynomial of $M$ is
  \begin{align*}
    (x-m_{00})^2 - \frac{\vec\nu(0)^2}{\vec\nu(1)^2}m_{01}^2,
  \end{align*}
  which has roots $m_{00}\pm \frac{\vec\nu(0)}{\vec\nu(1)}m_{01}$.
  The larger of these is  $m_{00}+ \frac{\vec\nu(0)}{\vec\nu(1)}m_{01}$.
  We then compute
  \begin{align*}
    \E[\ell_1(\Tpiv)] = \E[Z_0+Z_1] &= \vec\nu(0)(m_{00}+m_{01}) + \vec\nu(1)(m_{10}+m_{11})\\
      &= \vec\nu(0)(m_{00}+m_{01}) + \vec\nu(1)\Bigl(\tfrac{\vec\nu(0)^2}{\vec\nu(1)^2}m_{01}+m_{00}\Bigr)\\
      &= \bigl(\vec\nu(0)+\vec\nu(1)\bigr)m_{00} + \tfrac{\vec\nu(0)\bigl(\vec\nu(1) + \vec\nu(0)\bigr)}{\vec\nu(1)}m_{01}\\
      &= m_{00} + \tfrac{\vec\nu(0)}{\vec\nu(1)}m_{01}.\qedhere
  \end{align*}
\end{proof}

\begin{proof}[Proof of \ref{i:gen.size}]
  The value of $\E[ \ell_n(\Tpiv)]$ is the sum of entries
  of the vector $\vec\nu M^n$. We can confirm by hand that $\vec\nu$ is a left eigenvector
  of $M$ corresponding to the eigenvalue $\E[ \ell_1(\Tpiv)]$, from which the statement follows.
  
  There is a more conceptual explanation for this, which we briefly sketch. 
  Let $v$ be a vertex at level~$n$ of $T$,
  and consider the following question: conditional on $T|_n$ and on $v$ being pivotal,
  what is the distribution of $\omega(v)$? The answer is $\vec\nu$, just as if
  we had not conditioned on $v$ being pivotal. This is because switching the colour of $v$
  yields a bijection between colourings in which $v$ is pivotal with colour~$0$ and
  pivotal with colour~$1$, with a ratio $\vec\nu(0)/\vec\nu(1)$ of probabilities
  of each corresponding state under the product measure
  $\vec\nu^{\otimes \ell_n}$.
  Thus, pivotal vertices are coloured by $\vec\nu$, and so the expected number of pivotal
  children of a pivotal vertex is $\E[ \ell_1(\Tpiv)]$. Iterating this and applying
  linearity of expectation yields $\E[\ell_n(\Tpiv)]=\E[\ell_1(\Tpiv)]^n$.
\end{proof}
\begin{proof}[Proof of \ref{i:supercrit}]
  We consider two cases.
  First, suppose that $m_{01}=m_{10}=0$. By Lemma~\ref{lem:growth.rates},
  the matrix $M$ has the form $\bigl[\begin{smallmatrix}m_{00}&0\\0&m_{00}
  \end{smallmatrix}\bigr]$, and by our supercriticality assumption $m_{00}\geq 1$.
  Hence, $\Tpiv$ conditional on either $\omega(R_T)=0$ or $\omega(R_T)=1$ is a 
  supercritical single-type Galton--Watson tree, and it survives in both cases with
  positive probability.
  
  Now, suppose it is not true that $m_{01}=m_{10}=0$. 
  Since the multitype Galton--Watson tree $(\Tpiv,\omega)$ is supercritical, it survives with
  positive probability from some starting state.
  Hence at least one of the two probabilities $\P\bigl[\Tpiv \text{ survives}\bigmid \omega(R_{T}) = 0\bigr]$
  and $\P\bigl[\Tpiv \text{ survives}\bigmid \omega(R_{T}) = 1\bigr]$ must be positive. 
  By Lemma~\ref{lem:growth.rates},
  both $m_{01}$ and $m_{10}$ are positive. Thus, the root of $\Tpiv$ conditioned to be type~0
  has positive probability of giving birth to a pivotal vertex of type~$1$, and vice versa. 
  Therefore if either
  of $\P\bigl[\Tpiv \text{ survives}\bigmid \omega(R_{T}) = 0\bigr]$
  or $\P\bigl[\Tpiv \text{ survives}\bigmid \omega(R_{T}) = 1\bigr]$ is positive,
  then both of them are.
\end{proof}
\begin{proof}[Proof of \ref{i:crit}]
  As in the previous proof, we break the proof into two cases depending
  on whether $m_{01}=m_{10}=0$. If so, then $\Tpiv$ conditional on either $\omega(R_T)=0$ or $\omega(R_T)=1$
  is a critical single-type Galton--Watson tree, which dies out with probability one unless
  it is singular.
  To rule this out
  suppose that a vertex of type~$0$ gives birth to a single pivotal vertex of type~$0$
  with probability one.
  Then in particular, a vertex of type~$0$ always gives birth to exactly
  one child of type~$0$, since all children of a given type have the same pivotal status.
  Now, we claim that a vertex of type~$0$ cannot give birth to any vertices of type~$1$. Indeed,
  they would be nonpivotal, and hence switching one of them would yield another configuration
  with multiple children of type~$0$ but still with a type~$0$ root.
  (Note that we have assumed that $\vec\nu$ puts positive probability on both types, meaning
  that the configuration after the switching still has positive probability of occurring.)
  Hence, a vertex of type~$0$ gives birth almost surely to exactly one child, which has
  type~$0$. Thus, we have deduced the automaton: it assigns a parent type~$0$ 
  if and only if there is exactly one child, which has type~$0$. 
  Since $\vec\nu$ is a fixed point, it satisfies
  $  \vec\nu(0) = \chi(1)\vec\nu(0)$.
  But then $\vec\nu(0)\in\{0,1\}$, contradicting our assumption that $\vec\nu$
  places positive probability on both types. The same argument also shows that a vertex of type~$1$
  does not give birth to exactly one child of type~$1$ in the $m_{01}=m_{10}=0$ case.
  
  Now, consider the case that $m_{01}$ and $m_{10}$ are nonzero.
  According to \cite[Theorem~10.1]{Harris}, we must show that for the pivot tree, there does not
  exist a collection of states $\Cc$ such that the children of a vertex of type in $\Cc$
  almost surely include exactly one with type in $\Cc$.
  Suppose there exists such a set $\Cc$.
  If $\Cc=\{0\}$, then $m_{00}=1$. But as the highest eigenvalue of $M$
  is $m_{00}+ \frac{\vec\nu(0)}{\vec\nu(1)}m_{01}$ and $m_{01}$ is assumed to be nonzero,
  $\Tpiv$ is not critical.
  The same argument rules out $\Cc=\{1\}$.
  If $\Cc=\{0,1\}$, then every vertex (of whatever type)
  gives birth to exactly one pivotal vertex almost surely.
  Since all children of the same type have the same pivotality status, this implies that
  every vertex must give birth almost surely to a unique child 
  (i.e., one whose type is the opposite of all of its siblings).
  But this can happen only if $\chi$ is supported on $\{0,1\}$, since otherwise choosing the number
  of children according to $\chi$ and then colouring them i.i.d.\ $\vec\nu$, there is positive probability
  that they all are coloured the same. But this is a contradiction, since $\chi$
  is assumed to assign positive weight to $\{2,3,\ldots\}$.
\end{proof}

\section{Subcritical pivot trees}\label{sec:subcritical}
As in Section~\ref{sec:2}, throughout this section we fix
a child distribution $\chi$, an automaton $A$ on a finite set of states $\Sigma$, 
and a fixed point $\vec{\nu}$ of the automaton distributional map $\Psi\colon D\to D$ corresponding to $A$ and $\chi$. We let $(T,\omega)$ be the random state tree for $\vec\nu$. 
As usual, we let $\Tpiv$ denote the pivot tree for $(T,\omega)$, but in this section
we fix the maximal target set $\Aa_\sigma=\Sigma\setminus\{\sigma\}$ for
$\sigma\in\Sigma$. 
Throughout this section, when we refer to a vertex as
pivotal for $(T,\omega)$, we mean that it is pivotal with this target set
(see Section~\ref{pivot tree in subcritical}).
Recall from Proposition~\ref{prop:Tpiv.GW} that $\Tpiv$ is a Galton--Watson tree 
with the types defined there.
Our goal in this section is to prove the following:
\begin{prop}\label{prop:subcritical.implies.interpretable}
  Suppose that $\Tpiv$ is almost surely finite and that $\E \ell_n(\Tpiv)\leq 1$
  for all sufficiently large $n$.
  Then $\vec\nu$ is interpretable.
\end{prop}
This condition on $\Tpiv$ holds when it is subcritical,
and when $\abs{\Sigma}=2$ it also holds when $\Tpiv$ is critical, as discussed in
 Section~\ref{sec:pivot.regularity}.

Our proof will use the theory of Boolean functions and influences (see \cite{GS} and \cite{O'Donnell}).
We first introduce some ideas and results from this theory, starting with pivotality
in the context of Boolean functions.
For a function $g\colon\Sigma^m\to\{0,1\}$, we say that the $i$th
coordinate is pivotal for $g$ at $(s_1,\ldots,s_m)$ if the map
\begin{align*}
  s\mapsto g(s_1,\ldots,s_{i-1},s,s_{i+1},\ldots,s_m)
\end{align*}
is nonconstant. 
To relate this to our earlier definition of a pivotal vertex
in Section~\ref{pivot tree in subcritical}, 
recall the map $A^n_t\colon\Sigma^{\ell_n(t)}\to\Sigma$ defined in Section~\ref{sec:equivalent.conditions},
which gives the colour at the root of $t$ according to the automaton $A$ 
as a function of the colours at level~$n$.
For some fixed $\sigma\in\Sigma$, define $g_{t,n}\colon\Sigma^{\ell_n(t)}\to\{0,1\}$ by 
  \begin{align}\label{eq:gdef}
    (\sigma_1,\ldots,\sigma_{\ell_n(t)})\mapsto \1\big\{A^n_t(\sigma_1,\ldots,\sigma_{\ell_n(t)})=\sigma\big\}.
  \end{align}
 Then every pivotal coordinate for $g_{T,n}$ at $(\omega(v))_{v\in L_n(T)}$ is a pivotal vertex for 
 $(T,\omega)$. 
 We mention that the converse is false: not every pivotal vertex for $\omega$ is a pivotal coordinate, because changing the label of the vertex
  might change the label of the root from one element of $\Sigma\setminus\{\sigma\}$
  to another, leaving $g_{T,n}$ the same either way.

  The \emph{influence} of the $i$th coordinate of a map $g\colon\Sigma^m\to\{0,1\}$, 
  denoted by $I_i(g)$, is 
  the probability that the $i$th coordinate is pivotal for $(S_1,\ldots,S_m)$, where $S_1,\ldots,S_m$
  are independent and identically distributed as $\vec\nu$. The total influence, $I(g)$, is the sum of the influences of all
  the coordinates, or equivalently the expected number of pivotal coordinates for $g$ at $(S_1,\ldots,S_m)$.
  
  The following is a variant of the BKKKL inequality \cite[Theorem~1]{BKKKL}, which is
  itself a variant of the KKL inequality \cite{KKL}.
  \begin{prop}[Theorem~3.4 from \cite{FK}]\label{prop:BKKKL}
    There exists a universal constant $c>0$ such that the following holds.
    Let $g\colon\Sigma^n\to\{0,1\}$ be an arbitrary map, and let
    $p=\P[g(S_1,\ldots,S_n)=1]$, where $S_1,\ldots,S_n$ are independent and distributed
    as $\vec\nu$. Then
    \begin{align*}
      I(g)\geq c\min(p,1-p)\log\biggl(\frac{1}{\max_iI_i(g)}\biggr).
    \end{align*}
  \end{prop}
  Thus, if the total and maximum influences are small, then $\min(p,1-p)$ is small, meaning thta
  that $g$ is nearly constant. Our idea is to apply this to the map $g_{T,n}$ introduced
  in \eqref{eq:gdef}, which will then show that criterion~\ref{i:limit.interpretable}
  of Proposition~\ref{prop:interpretability} is satisfied and hence $\vec\nu$ is
  interpretable.
  
  For the rest of this section, we fix an arbitrary state $\sigma\in \Sigma$ and consider $g_{t,n}$ 
  as defined in \eqref{eq:gdef}.
  Define
  \begin{align*}
    I_n(t)=\E\bigl[\ell_n(\Tpiv)\bigmid T|_n=t|_n\bigr].
  \end{align*}
  When we consider the random state tree $(T,\omega)$ up to level~$n$, there are two sources of
  randomness: the tree itself, which is Galton--Watson, and the colours, which are determined
  by colouring the level~$n$ vertices i.i.d.\ $\vec\nu$. We obtain $I_n(t)$ by taking an expectation
  only over this second source of randomness, with the structure of the tree fixed.  
  In other words, if the level~$n$ vertices of the deterministic tree $t$ are coloured i.i.d.\ $\vec\nu$,
  then $I_n(t)$ is the expected number of these vertices that are pivotal. Thus, $I_n(T)$ is
  the expected number of pivotal vertices for $(T,\omega)$ conditional on $T|_n$.
  Since a level~$n$ vertex of $T$ is pivotal for $\omega$ if
  the corresponding coordinate of $g_{T,n}$ is pivotal at $(\omega(v))_{v\in L_n(T)}$, 
  we have $I(g_{T,n})\leq I_n(T)$.

  For a given tree~$t$, let
  \begin{align*}
    \Imax_n(t) = \max_{v\in L_n(t)} \P\bigl[v \in \Tpiv\bigmid T|_n=t|_n\bigr].
  \end{align*}
  Observe that $I_n(t)$ has the same definition
  except that a sum replaces the maximum. Just as $I(g_{T,n})\leq I_n(T)$, we have
  $\max_i I_i(g_{T,n})\leq \Imax_n(T)$.
  \begin{lemma}\label{lem:Imax}
    If $\Tpiv$ is almost surely finite, then $\Imax_n(T)\to 0$ a.s.\ as $n\to\infty$.
  \end{lemma}
  \begin{proof}
    We will show this by proving that
    \begin{align} \label{eq:max.claim.1}
      \Imax_n(T)\leq \P\bigl[ \text{$\Tpiv$ survives to height~$n$} \bigmid T|_n \bigr]\text{ a.s.}
    \end{align}
    and
    \begin{align}\label{eq:max.claim.2}
      \P\bigl[ \text{$\Tpiv$ survives to height~$n$} \bigmid T|_n \bigr]\to 0 \text{ a.s.}
    \end{align}
    as $n\to\infty$.
    
    For the first claim, we start with the observation that for any $v\in L_n(T)$,
    \begin{align*}
      \P\bigl[ v\in\Tpiv\bigmid T|_n\bigr] \leq 
         \P\bigl[ \text{$\Tpiv$ survives to height~$n$} \bigmid T|_n \bigr]\text{ a.s.,}
    \end{align*}
    since $v\in\Tpiv$ implies that $\Tpiv$ survives to height~$n$. Since
    \begin{align*}
      \Imax_n(T) &= \max_{v\in L_n(T)} \P\bigl[ v\in\Tpiv\bigmid T|_n\bigr],
    \end{align*}
    this proves \eqref{eq:max.claim.1}.
    
    Now we turn to \eqref{eq:max.claim.2}. As $n\to\infty$,
    \begin{align*}
      \P[\text{$\Tpiv$ survives to height~$n$}]\to 0,
    \end{align*}
    since $\Tpiv$ is almost surely finite. Hence the convergence
    in \eqref{eq:max.claim.2} holds in $L^1$. To get the almost sure convergence, we show that
    \begin{align}\label{eq:smtg}
      \P\bigl[ \text{$\Tpiv$ survives to height~$n$} \bigmid T|_n \bigr]
    \end{align}
    is a supermartingale, which is more trivial than it looks
    at first glance. If $\Tpiv$ survives to height $n+1$, then it survives to height $n$. Hence,
    \begin{align*}
      \P\bigl[ \text{$\Tpiv$ survives to height~$n+1$} \bigmid T|_n,\omega|_{T|_n} \bigr]
        &\leq \P\bigl[ \text{$\Tpiv$ survives to height~$n$} \bigmid T|_n,\omega|_{T|_n} \bigr]\\
        &= \1\{\text{$\Tpiv$ survives to height~$n$}\}.
    \end{align*}
    Taking conditional expectations,
    \begin{align*}
      \P\bigl[ \text{$\Tpiv$ survives to height~$n+1$} \bigmid T|_n \bigr]
        &\leq \P\bigl[ \text{$\Tpiv$ survives to height~$n$} \bigmid T|_{n} \bigr].
    \end{align*}
    Finally,
    \begin{align*}
      &\E\Bigl[\P\bigl[ \text{$\Tpiv$ survives to height~$n+1$} \bigmid T|_{n+1} \bigr]\Bigmid T|_n\Bigr] \\
        &\qquad\qquad= \P\bigl[ \text{$\Tpiv$ survives to height~$n+1$} \bigmid T|_n \bigr].
    \end{align*}
    Altogether, this shows that
    \begin{align*}
      \E\Bigl[\P\bigl[ \text{$\Tpiv$ survives to height~$n+1$} \bigmid T|_{n+1} \bigr]\Bigmid T|_n\Bigr]
        &\leq \P\bigl[ \text{$\Tpiv$ survives to height~$n$} \bigmid T|_{n} \bigr],
    \end{align*}
    proving that \eqref{eq:smtg} is a supermartingale. 
    Thus it has an almost sure limit, which must coincide with the $L^1$ limit. 
    This proves \eqref{eq:max.claim.2}, which completes the proof.
  \end{proof}
  
  Next, we give two easy technical lemmas to be used in the proof of 
  Proposition~\ref{prop:subcritical.implies.interpretable}.
  \begin{lemma}\label{lem:tech1}
    Let $X_n$ and $Y_n$ be nonnegative random variables, and suppose that
    $\E X_n\leq 1$ for all $n$ and $Y_n\to\infty$~a.s.
    Let $Z_n=\min(X_n/Y_n, 1)$. Then $\E Z_n\to 0$.
  \end{lemma}
  \begin{proof}
    Fix some large $N$. We then compute
    \begin{align*}
      \E Z_n &= \E\bigl[ Z_n\1\{Y_n\geq N\}\bigr] + \E\bigl[ Z_n\1\{Y_n< N\}\bigr]\\
         &\leq \E\bigl[ X_n/N \bigr] + \P\bigl[  Y_n< N \bigr]\\
         &\leq 1/N + \P\bigl[  Y_n< N \bigr].
    \end{align*}
    Since $\P[  Y_n< N ]\to 0$ as $n\to\infty$, we have 
    $\limsup_{n\to\infty}\E Z_n\leq 1/N$. This holds for arbitrarily large values
    of $N$, confirming that $\E Z_n\to 0$.
  \end{proof}
  
  \begin{lemma}\label{lem:tech2}
    Suppose that $(X_n)_{n\geq 0}$ takes values in $[0,1]$ and is a martingale 
    under some filtration $(\Ff_n)_{n\geq 0}$.
    Then $\bigl(\min(X_n, 1-X_n)\bigr)_{n\geq 0}$ is a supermartingale under the same filtration.
  \end{lemma}
  \begin{proof}
    Let $Y_n=\min(X_n,1-X_n)$. Since $Y_{n+1}\leq X_{n+1}$ and $Y_{n+1}\leq 1-X_{n+1}$,
    we can take expectations to get
    \begin{align}\label{eq:X_n}
      \E[Y_{n+1}\mid \Ff_n] &\leq \E[ X_{n+1}\mid \Ff_n] = X_n,\\\intertext{and}
      \E[Y_{n+1}\mid \Ff_n] &\leq[ 1-\E X_{n+1}\mid \Ff_n] = 1-X_n.\label{eq:1-X_n}
    \end{align}
    On the event $X_n\leq 1/2$, which is measurable with respect to $\Ff_n$, equation~\eqref{eq:X_n} gives
    $\E[Y_{n+1}\mid\Ff_n]\leq Y_n$, since $Y_n=X_n$. On the complement of this event,
    \eqref{eq:1-X_n} gives $\E[Y_{n+1}\mid\Ff_n]\leq Y_n$, since $X_n=1-Y_n$.
    Hence $\E[Y_{n+1}\mid\Ff_n]\leq Y_n$ holds in both cases, proving that $(Y_n)$ is a supermartingale.    
  \end{proof}

  \begin{proof}[Proof of Proposition~\ref{prop:subcritical.implies.interpretable}]
    We will check that criterion~\ref{i:limit.interpretable} of Proposition~\ref{prop:interpretability} 
    holds. Fix a colour $\sigma$ in $\Sigma$. Let 
    \begin{align*}
      p(T|_n)&= \P\bigl[\omega(R_T)=\sigma\bigmid T|_n\bigr].
    \end{align*}
    Our goal is to show that $p(T|_n)$ converges almost surely to $0$ or $1$.
    We can assume that $\sigma$ is in the support of $\vec\nu$, since otherwise
    $p(T|_n)=0$ a.s.\ for all $n$.    
    Consider $g_{T,n}$ as defined in \eqref{eq:gdef}.
    Observe that
    \begin{align*}
      \1\{\omega(R_T)=\sigma\} &= g_{T,n}\bigl( (\omega(v))_{v\in L_n(T)} \bigr),
    \end{align*}
    and that the conditional distribution of $(\omega(v))_{v\in L_n(T)}$ given $T|_n$ is i.i.d.\ $\vec\nu$.
    We thus apply Proposition~\ref{prop:BKKKL} conditionally on $T|_n$ to obtain
    \begin{align*}
      I(g_{T,n}) \geq c\min\bigl(p(T|_n),\,1-p(T|_n)\bigr)\log\biggl(\frac{1}{\max_i I_i(g_{T,n})}\biggr).
    \end{align*}
    Rearranging this, we obtain
    \begin{align}\label{eq:BKKKL.consequence}
      \min\bigl(p(T|_n),\,1-p(T|_n)\bigr)\leq\frac{I(g_{T,n})}{c\log\frac{1}{\max_i I_i(g_{T,n})}}
        \leq \frac{I_n(T)}{c\log\frac{1}{\Imax_n(T)}}.
    \end{align}

    Now, we show that $\min\bigl(p(T|_n),\,1-p(T|_n)\bigr)$ converges to $0$
    in $L^1$ as $n\to\infty$. Let $X_n=I_n(T)$ and $Y_n=c\log\frac{1}{\Imax_n(T)}$. Then $\E X_n$ is
    the expected number of pivotal vertices for $(T,\omega)$ at level~$n$, which by assumption
    is bounded by $1$. Since $\Tpiv$ is assumed to be almost surely finite,
    Lemma~\ref{lem:Imax} shows $Y_n\to\infty$ a.s.
    By Lemma~\ref{lem:tech1}, it holds that $\E\min(X_n/Y_n,1)\to 0$.
    Since $\min\bigl(p(T|_n),\,1-p(T|_n)\bigr)\leq\min(X_n/Y_n,1)$, this proves that
    $\min\bigl(p(T|_n),\,1-p(T|_n)\bigr)$ converges to $0$ in $L^1$.
    
    Since $p(T|_n)$ is a martingale, Lemma~\ref{lem:tech2} shows that
    $\min\bigl(p(T|_n),\,1-p(T|_n)\bigr)$ is a supermartingale. Hence it has an almost
    sure limit, which must match its $L^1$ limit of $0$. This proves that $\lim p(T|_n)\in\{0,1\}$ a.s.
    By Proposition~\ref{prop:interpretability}, the fixed point $\vec\nu$ is intepretable.
  \end{proof}
  \begin{proof}[Proof of Theorem~\ref{main 2 colours} ($\Longleftarrow$)]
    Subcritical Galton--Watson trees have exponentially
    vanishing expected $n$th generation size and are almost surely finite
    (see Section~\ref{sec:pivot.regularity}). Hence the conditions of
    Proposition~\ref{prop:subcritical.implies.interpretable} hold in the
    subcritical case.
    In the critical case, $\E\ell_n(\Tpiv)=1$ and $\Tpiv$ is almost surely finite
    by parts~\ref{i:gen.size} and \ref{i:crit} of Proposition~\ref{prop:2state.regularity}.
  \end{proof}
  \begin{proof}[Proof of Theorem~\ref{thm:k.state}]
    If $\Tpiv$ is subcritical, then the conditions of Proposition~\ref{prop:subcritical.implies.interpretable}
    hold.
  \end{proof}
  
  \begin{remark}\label{rmk:critical.case}
    While we have stated Theorem~\ref{thm:k.state} for subcritical pivot trees only,
    Proposition~\ref{prop:subcritical.implies.interpretable} in fact applies to critical pivot trees,
    so long as $\E \ell_n(\Tpiv)\leq 1$ (any constant
    bound would also work). 
    As discussed in Section~\ref{sec:pivot.regularity}, 
    Galton--Watson trees that are not positive regular can
    have their expected $n$th generation size grow to infinity even in the critical case.
    However, even though pivot trees are not necessarily positive regular,
    we are not sure if it is possible for $\E\ell_n(\Tpiv)$ to grow to infinity when $\Tpiv$
    is critical. 
  \end{remark}

\section{Supercritical pivot trees}
\label{sec:supercritical}
As usual, throughout this section we fix
a child distribution $\chi$, an automaton $A$ on a set of states $\Sigma$, 
and a fixed point $\vec{\nu}$ of the automaton distributional map $\Psi\colon D\to D$ 
corresponding to $A$ and $\chi$, and
we let $(T,\omega)$ be the random state tree for $\vec\nu$.

Our goal is to prove that if the pivot tree is supercritical, then $\vec\nu$
is rogue in the $\abs{\Sigma}=2$ case.
According to Proposition~\ref{prop:interpretability}, rogueness of $\vec\nu$ is equivalent
to nonmeasurability of $\omega$ with respect to $T$. Thus, we will try to
show that for $T$ in some class of trees of positive weight under the $\GW(\chi)$
measure, the colouring $\omega$ is nondeterministic.
The idea of the proof is that $\Tpiv$ is supercritical,
it occurs with positive probability that $\omega(R_T)=0$ and $\Tpiv$ survives.
On this event, we randomly choose an infinite path starting from the root of $\Tpiv$
and switch all the colours along it. This gives us a new coloured tree with the same underlying tree
but a different colour at the root. Since the new colouring of the tree only differs at one
vertex per level, it seems intuitive that it occurs with similar likelihood as the original one,
meaning that $\omega$ takes different values for the same tree with positive probability.

The difficulty lies in making rigorous the idea that the switched colouring has similar
probability as the original one. To do so, we use \emph{spine decompositions} as developed
by Lyons, Pemantle, and Peres \cite{LPP}, an elegant
probabilistic method for proving two branching processes absolutely continuous
or mutually singular to each other. 
The two processes we consider are $(T,\omega)$, conditioned on survival of the pivot tree,
and the switched version of this process described above.
We prove the switched version is
absolutely continuous with respect to the original. 
Under the assumption that $\vec\nu$ is interpretable, it is a probability one event that
the colour $\omega(R_T)$ is given
as a deterministic function of $T$. By absolute continuity, the colour of the root in the switched
process is equal to the same function of the tree. But this is a contradiction,
as we know that these colours differ while the trees are the same.

Let our set of colours be $\Sigma=\{0,1\}$.
Recall from the discussion after Proposition~\ref{prop:Tpiv.GW} that the definition of the
pivot tree is simpler in the two-colour case. A vertex is
pivotal for $(T,\omega)$ if swapping its colour results in the root swapping colours, and the pivot
tree $\Tpiv$ can be defined as the subtree of $T$ consisting of all pivotal vertices.
The pivot tree is Galton--Watson with types given by $\omega$, with no need for the sets
$\Aa$ and $B_v$ used in the definitions when there are three or more colours.

We now formalize this and add an extra bit of information to the types, extending $\omega$
to a map $\omega_*\colon V(T)\to \{\zd,\zs,\od,\os\}$
as follows.
For a vertex $v$, the 0 or 1 in $\omega_*(v)$ is given by $\omega(v)$. To decide on {d} or {s},
consider $(T(v),\omega|_{T(v)})$, the restriction of the random state tree to $v$ and its
descendants. If this tree has an infinite pivot tree, then $\omega_*(v)$ assigns type {s},
for \emph{survives}. If this tree has a finite pivot tree, then $\omega_*(v)$ assigns type {d},
for \emph{dies}. For $v\in\Tpiv$, this is equivalent to assigning either {s} or {d} depending
on whether $\Tpiv$ restricted to $v$ and its descendants is infinite or finite.
We will refer to vertices as s-labeled or d-labeled according to the value assigned to them
by $\omega_*$.
 Define $\vec{\nu}_*$ as the distribution of $\omega_*(R_T)$, a probability measure on $\{\zd,\zs,\od,\os\}$.
Let $\Tcolstar$ denote the space of trees with vertices labeled $\{\zd,\zs,\od,\os\}$.
For $(t,\tau_*)\in\Tcolstar$, let $[t,\tau_*]_n$ denote the subset of $\Tcolstar$ made up of trees
agreeing with $(t,\tau_*)$ up to the $n$th generation.
 
\begin{prop}\label{prop:(T,omega*)}\ 
  \begin{enumerate}[(a)]
    \item Conditional on $T|_n$, the distribution of $(\omega_*(v))_{v\in L_n(T)}$ is
      i.i.d.\ $\vec{\nu}_*$.\label{i:star1}
    \item 
      For $i\in\{0,1\}$, let $\rho(i)$ be the probability that $\Tpiv$ survives
      conditional on $\omega(R_T)=i$. Then conditional on $T|_n$ and on $\omega_{T|_n}$,
      the s- and d-labels given to each vertex $v\in L_n$ by $\omega_*$ are independent,
      with $v$ receiving an s-label with probability $\rho(\omega(v))$.
      \label{i:star2}
    \item
      The labeled tree $(T,\omega_*)$ is multitype Galton--Watson.
      \label{i:star3}
  \end{enumerate}
\end{prop}
\begin{proof}
  Given $T|_n$, the distribution of $(\omega(v))_{v\in L_n(T)}$ is
  i.i.d.\ $\vec{\nu}$, by definition of $(T,\omega)$.
  Hence, conditional on $T|_n$, the trees $(T(v),\omega|_{T(v)})$ for $v\in L_n(T)$
  are independent and distributed as the (unconditional) distribution of $(T,\omega)$. 
  Since $\vec{\nu}_*$ is the distribution
  of $\omega_*(R_T)$, it is thus the conditional distribution given $T|_n$ of each of 
  the independent $\omega_*(v)$ for $v\in L_n(T)$, proving \ref{i:star1}.
  For \ref{i:star2}, once we have conditioned on $T|_n$ and on $\omega|_{T|_n}$,
  for each $v\in L_n(T)$, the tree $(T(v),\omega|_{T(v)})$ is distributed as
  $(T,\omega)$ conditional on having state $\omega(v)$ at the root. Thus the pivot
  tree of $(T(v),\omega|_{T(v)})$ survives with probability $\rho(\omega(v))$.  
  The s- or d-label
  for $v$ depends only on $(T(v),\omega|_{T(v)})$ and hence are given independently.
  
  The proof of \ref{i:star3} is nearly the same
  as the proof of Proposition~\ref{(T, omega) multi-type}, though we will give it
  now in detail.
  For $\sigma_1,\ldots\sigma_k\in\{\zd,\zs,\od,\os\}$, let
  \begin{align*}
    \chicolstar(\sigma_1,\ldots,\sigma_k)=\chi(k)\vec\nu_*(\sigma_1)\cdots\vec\nu_*(\sigma_k).
  \end{align*}
  By the first claim, this is the probability that $R_T$ has exactly $k$ children
  whose types according to $\omega_*$ are $\sigma_1,\ldots,\sigma_k$, in order.
  For any type $\sigma\in\{\zd,\zs,\od,\os\}$ with $\vec\nu_*(\sigma)>0$, let
  $\chicolstar^\sigma(\sigma_1,\ldots,\sigma_k)$ denote the conditional probability
  that $R_T$ gives birth to $k$ children of types $\sigma_1,\ldots,\sigma_k$ according to $\omega_*$
  given that $\omega_*(R_T)=\sigma$. Observe that the value of $\omega_*$ at the root
  of a tree can be determined from the value of $\omega_*$ at its children: the $0$ or $1$ can be
  determined according to the automaton, and the s- or d-type can be determined according to whether
  there is a pivotal child of the root of s-type. Hence, if $\sigma$ is the type at the root
  corresponding to children of types $\sigma_1,\ldots,\sigma_k$, then
  \begin{align}\label{eq:chistar.condition}
    \chicolstar(\sigma_1,\ldots,\sigma_k)=\vec\nu_*(\sigma)\chicolstar^\sigma(\sigma_1,\ldots,\sigma_k).
  \end{align}
  
  Now, we seek to prove that given the first $n$ levels of $(T,\omega_*)$, each vertex~$v$ at
  level~$n$ independently gives birth according to $\chicolstar^{\omega_*(v)}$.
  Fix $(t,\tau_*)\in\Tcolstar$. 
  By the first claim of this proposition,
  \begin{align*}
    \P\Bigl[ (T,\omega_*)\in[t,\tau_*]_{n+1} \Bigr] = \P\Bigl[T\in[t]_{n+1}\Bigr] 
                                                 \prod_{u\in L_{n+1}(t)}\vec\nu_*\bigl(\tau_*(u)\bigr).
  \end{align*}
  With $C(v)$ denoting the set of a children of a vertex~$v$,
  \begin{align*}
        \P\Bigl[ (T,\omega)\in[t,\tau_*]_{n+1} \Bigr]
      &=\Biggl(\P\Bigl[T\in[t]_n\Bigr]\prod_{v\in L_n(t)}\chi\bigl(\abs{C(v)}\bigr)\Biggr)\prod_{u\in L_{n+1}(t)}\vec\nu_*\bigl(\tau_*(u)\bigr)\\
      &= \P\Bigl[T\in[t]_n\Bigr]\prod_{v\in L_n(t)}\chicolstar\bigl(\tau_*(u)_{u\in C(v)}\bigr).
  \end{align*}
  Continuing to follow the proof of Proposition~\ref{(T, omega) multi-type}, by
  \eqref{eq:chistar.condition},
  \begin{align*}
    \P\Bigl[ (T,\omega)\in[t,\tau_*]_{n+1} \Bigr]
      &= \P\Bigl[T\in[t]_n\Bigr]\prod_{v\in L_n(t)}\vec\nu_*(\tau_*(v))\chicolstar^{\tau_*(v)}\bigl(\tau_*(u)_{u\in C(v)}\bigr)\\
      &= \P\Bigl[(T,\omega)\in[t,\tau_*]_n\Bigr]\prod_{v\in L_n(t)}\chicolstar^{\tau_*(v)}\bigl(\tau_*(u)_{u\in C(v)}\bigr). \qedhere
  \end{align*}
\end{proof}

We will assume throughout the section that $\Tpiv$ is supercritical. 
This implies that either $\P[\omega_*(R_T)=\zs]>0$ or
$\P[\omega_*(R_T)=\os]>0$, but in fact both are true by 
Proposition~\ref{prop:2state.regularity}\ref{i:supercrit}.
Thus, it makes sense to consider the distribution of $(T,\omega_*)$
conditional on $\omega_*(R_T)=\zs$ or $\omega_*(R_T)=\os$.
With this in mind, we make a number of definitions.
Most important among them are the probability measures $\RST^\zs$, $\RST^\os$, $\RST^{\zs\to\os}$,
and $\RST^{\os\to\zs}$ on the space $\Tcolstar$, with $\RST$ standing for \emph{random state tree}. 
The measures $\RST^\zs$ and $\RST^\os$ are the distributions of
$(T,\omega_*)$ conditioned on $\omega_*(R_T)=\zs$ and $\omega_*(R_T)=\os$, respectively.
The measure $\RST^{\zs\to\os}$ is the distribution of a labeled tree obtained by sampling
from $\RST^\zs$, choosing an infinite path of pivotal vertices, and swapping every label
in the path. The measure $\RST^{\os\to\zs}$ is obtained in the same way, starting with $\RST^\os$
instead of $\RST^\zs$. Thus, $\RST^\os$ and $\RST^{\zs\to\os}$ are both distributions on labeled
trees with $\os$ at the root. Our goal, as we sketched before and will explain in more
detail shortly, is to prove that 
$\RST^{\zs\to\os}$ is absolutely continuous with respect to $\RST^\os$.

\begin{defines}[{Definitions of 
  $\RST^\ell$, $(T^\ell,\omega_*^\ell)$, $\Ts$, $\Ww(t,\tau_*)$, $\Ww_n(t,\tau_*)$, 
  $P_{\Ww(t,\tau_*)}$, $P_{\Ww_n(t,\tau_*)}$, $\tau_*^{\vec{v}}$, $\tau_*^x$,
  $\RST^{\zs\to\os}$, $\RST^{\os\to\zs}$, $(T^\zs,\omega_*^{\zs\to\os})$, and $(T^\os,\omega_*^{\os\to\zs})$}]\ \\
  For $\ell\in\{\zd,\od,\zs,\os\}$, let
  $\RST^\ell$ be the law of $(T,\omega_*)$ conditioned on $\omega(R_T)=\ell$.
  Let $(T^\ell,\omega^\ell_*)$ be a random variable distributed as $\RST^\ell$.
  Let $\Ts\subseteq\Tcolstar$ be the set of all trees $(t,\tau_*)$ labeled by
  $\{\zd,\od,\zs,\os\}$ that are compatible with the automaton, have their d
  and s labels consistent with the tree and other labels, and have $\zs$ or $\os$ at the root.
  This space is the union of the supports of $\RST^\zs$ and $\RST^\os$.
  It could also be defined as the set of all trees $(t,\tau_*)\in\Tcolstar$
  such that $[t,\tau_*]_n$ has positive probability
  under $\RST^\zs$ or $\RST^\os$ for all $n$.
  
  Given a deterministic tree $(t,\tau_*)\in\Ts$, let
  $\Ww(t,\tau_*)$ be the set of infinite paths in $(t,\tau_*)$ that start at $R_t$
  and contain only pivotal s-labeled vertices. Let $\Ww_n(t,\tau_*)$ be the set of paths
  from $R_t$ of length $n$ with the same property. Note that these sets are nonempty
  for any $(t,\tau_*)\in\Ts$, since any pivotal
  s-labeled vertex must have a pivotal s-labeled child.
  
  We define $P_{\Ww(t,\tau_*)}$ and $P_{\Ww_n(t,\tau_*)}$ to be distributions on
  $\Ww(t,\tau_*)$ and $\Ww_n(t,\tau_*)$, respectively, given as follows.
  Let $V_0=R_t$. Choose $V_1$ uniformly from the pivotal s-labeled children of $V_0$
  (as we said, there must be at least one). Then choose $V_2$ uniformly from
  the pivotal s-labeled children of $V_1$, and so on. Let  
  $P_{\Ww(t,\tau_*)}$ be the distribution of $(V_0,V_1,\ldots)$, and let
  $P_{\Ww_n(t,\tau_*)}$ be the distribution of $(V_0,\ldots,V_n)$.
  
  For an assignment $\tau\colon R(t)\to\Sigma$ and a vertex $v\in L_n(t)$, we defined
  $\tau^{v\to\gamma}$ as the colouring of $t|_n$ given by switching the colour of $v$
  to $\gamma$ and updating the colours at levels $0,\ldots,n-1$ according to the automaton.
  We now extend this definition to allow switching when the colours include s- and d-labels
  and to allow switching an infinite path.
  Suppose that $(t,\tau_*)\in\Tcolstar$.
  Given a path $\vec{u}=(u_0,u_1,\ldots)\in\Ww(t,\tau_*)$,
  let $\tau_*^{\vec{u}}$ denote $\tau_*$ with labels $\zs$ and $\os$ swapped along the path.
  It is easy to check that this new labeling is also compatible with the automaton~$A$, and
  that its s and d markings follow the same rules as before. For $\vec{u}\in\Ww_n(t,\tau_*)$,
  we define $\tau_*^{\vec{u}}$ in the same way, except that $\tau_*^{\vec{u}}$ is only a labeling
  of $t|_n$. For a vertex $x\in R(t)$, we use $\tau_*^x$
  as a shorthand for $\tau_*^{\vec{u}}$, where $\vec{u}$ is the path from $R_t$ to $x$.

  Finally, we define $\omega_*^{\zs\to\os}$ as the switched labeling $(\omega_*^\zs)^{\vec{V}}$,
  where $\vec{V}$ is sampled from $P_{\Ww(T^\zs,\omega_*^\zs)}$.
  To summarize, $(T^\zs,\omega_*^{\zs\to\os})$ is formed by the following procedure:
  First condition $(T,\omega_*)$ on $\omega_*=\zs$ to obtain $(T^\zs,\omega_*^\zs)$. Then, choose
  an infinite s-labeled path of pivotal vertices in $(T^\zs,\omega_*^\zs)$
  by starting at the root and successively choosing a pivotal s-labeled
  child at random. Finally, swap all $\zs$ and $\os$ labels along this path to obtain
  $(T^\zs,\omega_*^{\zs\to\os})$.
  We define $\omega_*^{\os\to\zs}$ analogously, and we define $\RST^{\zs\to\os}$
  and $\RST^{\os\to\zs}$ as the distributions of $(T^\zs,\omega_*^{\zs\to\os})$
  and $(T^{\os},\omega_*^{\os\to\zs})$, respectively.
\end{defines}
  We now lay out our plan for the section.  
  Our goal is to prove that $\RST^{\zs\to\os}\ll\RST^\os$.
  It follows quickly from this that $\vec\nu$ is rogue by an argument we sketch now.
  Supposing that $\vec\nu$ is interpretable, we can express $\omega(R_T)$ as $\iota(T)$ for a deterministic
  function $\iota\colon\Tt\to\Sigma$, by Proposition~\ref{prop:interpretability}.
  By definition of $\RST^\zs$ and $\RST^\os$, we have $\iota(T^\zs)=0$ a.s.\ and
  $\iota(T^\os)=1$ a.s. Recalling that $(T^\zs,\omega_*^{\zs\to\os})\sim\RST^{\zs\to\os}$, absolute continuity
  lets us conclude from $\iota(T^\os)=1$ a.s.\ that $\iota(T^\zs)=1$ a.s.,
  a contradiction.
  
  To prove the absolute continuity of $\RST^{\zs\to\os}$ with respect to $\RST^\os$,
  we use a technique of restricting these measures to successively larger $\sigma$-algebras
  and computing the Radon-Nikodym derivatives of the restricted measures.
  The result we use is well known:
  \begin{lemma}[{\cite[Lemma~12.2]{LP}}]\label{lem:dichotomy}
    Let $\mu$ and $\nu$ be probability measures on a $\sigma$-algebra $\Fff$.
    Suppose that $\Fff_1\subseteq\Fff_2\subseteq\cdots\subseteq\Fff$, and that $\cup_n\Fff_n$
    generates $\Fff$. Also suppose that $\mu|_{\Fff_n}$ is absolutely continuous with respect
    to $\nu|_{\Fff_n}$ with Radon-Nikodym derivative $X_n$. Define $X=\limsup_{n\to\infty} X_n$.
    Then
    \begin{align*}
      \mu\ll\nu \quad &\iff \quad X<\infty \text{ $\mu$-a.e.} \quad \iff \quad \int X\,d\nu = 1,\\
      \intertext{and}
      \mu\perp\nu \quad &\iff \quad X=\infty \text{ $\mu$-a.e.} \quad \iff \quad \int X\,d\nu=0.
    \end{align*}    
  \end{lemma}
  
  In our case, we will restrict $\RST^{\zs\to\os}$ and $\RST^\os$ to the $\sigma$-algebra
  $\Fff_n$ generated by the first $n$ levels of the labeled tree. That is, we define $\Fff_n$ as
  the $\sigma$-algebra on $\Ts$ generated by the sets of the form $[t,\tau_*]_n$.
  
  To investigate these Radon--Nikodym derivatives, we start by giving
  representations of $\RST^{\zs\to\os}$ and $\RST^{\os\to\zs}$
  in terms of $\RST^\zs$ and $\RST^\os$:
  \begin{lemma}\label{lem:RST.rep}
    For any $(t,\tau_*)\in\Ts$,
    \begin{align}
      \RST^{\zs\to\os}[t,\tau_*]_n &= \sum_{\vec{u}\in\Ww_n(t,\tau_*)}\RST^\zs[t,\tau_*^{\vec{u}}]_n P_{\Ww_n(t,\tau_*^{\vec{u}})}(\vec{u}),\label{eq:RST.rep.1}\\
      \intertext{and}
      \RST^{\os\to\zs}[t,\tau_*]_n &= \sum_{\vec{u}\in\Ww_n(t,\tau_*)}\RST^\os[t,\tau_*^{\vec{u}}]_n P_{\Ww_n(t,\tau_*^{\vec{u}})}(\vec{u}).\label{eq:RST.rep.2}
    \end{align}
  \end{lemma}
  \begin{proof}
    These statements follow very directly from the definitions.
    Recall that $(T^{\zs},\omega_*^{\zs\to\os})$ differs from $(T^\zs,\omega_*^\zs)$
    along a random path $(V_i)_{i\geq 0}$ sampled from $P_{\Ww(T^\zs,\omega_*^\zs)}$, in which
    all $\zs$ and $\os$ labels have been swapped. Hence,
    \begin{align*}
      (T^{\zs},\omega_*^{\zs\to\os})\in[t,\tau_*]_n
    \end{align*}
    holds if and only if
    \begin{align}\label{eq:rh.statement}
      (T^\zs,\omega_*^\zs)\in[t,\tau_*^{\vec{u}}]\text{ for some }\vec{u}\in\Ww_n(t,\tau_*^{\vec{u}}),
    \text{ and } (V_0,\ldots,V_n)=\vec{u}.
    \end{align}
    Since $\vec{u}\in\Ww_n(t,\tau_*^{\vec{u}})$ if
    and only if $\vec{u}\in\Ww_n(t,\tau_*)$, we can refine \eqref{eq:rh.statement} to
    \begin{align}\label{eq:rh.refined}
      (T^\zs,\omega_*^\zs)\in[t,\tau_*^{\vec{u}}]\text{ for some }\vec{u}\in\Ww_n(t,\tau_*),
    \text{ and } (V_0,\ldots,V_n)=\vec{u}.
    \end{align}
    As the events in \eqref{eq:rh.refined} are disjoint for different choices of $\vec{u}$,
    \begin{align*}
      \P\Bigl[ (T^{\zs},&\omega_*^{\zs\to\os})\in[t,\tau_*]_n \Bigr]\\
        &=\sum_{\vec{u}\in\Ww_n(t,\tau_*)} \P\Bigl[\text{$(T^\zs,\omega_*^\zs)\in[t,\tau_*^{\vec{u}}]_n$
                                       and $(V_0,\ldots,V_n)=\vec{u}$} \Bigr]\\
        &= \sum_{\vec{u}\in\Ww_n(t,\tau_*)} \P\Bigl[(T^\zs,\omega_*^\zs) \in [t,\tau_*^{\vec{u}}]_n\Bigr]P_{\Ww_n(t,\tau_*^{\vec{u}})}(\vec{u}),
    \end{align*}
    which is a restatement of \eqref{eq:RST.rep.1}. The proof of \eqref{eq:RST.rep.2} is identical.
  \end{proof}

  Define a map $r_n$ on $\Ts$ as follows.
  For $(t,\tau_*)\in\Ts$ with $\zs$ at the root, let
  \begin{align}\label{eq:r_n.def.1}
    r_n(t,\tau_*) &= \frac{\RST^{\os\to\zs}[t,\tau_*]_n}{\RST^\zs[t,\tau_*]_n},
  \end{align}
  and for $(t,\tau_*)\in\Ts$ with $\os$ at the root, let
  \begin{align}\label{eq:r_n.def.2}
    r_n(t,\tau_*) &= \frac{\RST^{\zs\to\os}[t,\tau_*]_n}{\RST^\os[t,\tau_*]_n}.
  \end{align}
  Thus, $r_n(t,\tau_*)$ matches the Radon--Nikodym derivative either of
  $\RST^{\os\to\zs}|_{\Fff_n}$ with respect to $\RST^{\zs}|_{\Fff_n}$ or of
  $\RST^{\zs\to\os}|_{\Fff_n}$ with respect to $\RST^\os|_{\Fff_n}$, depending
  on $\tau_*(R_t)$. 
  To prove that $\RST^{\zs\to\os}\ll\RST^\os$, it therefore suffices  
  by Lemma~\ref{lem:dichotomy} to show that
  \begin{align}\label{eq:limsup.finite}
    \limsup_{n\to\infty} r_n(T^\zs,\omega_*^{\zs\to\os})<\infty \text{ a.s.}
  \end{align}

  We define
  \begin{align*}
    f_n(t,\tau_*) &= \sum_{\vec{u}\in\Ww_n(t,\tau_*)} P_{\Ww_n(t,\tau_*^{\vec{u}})}(\vec{u}),
  \end{align*}
  with $f_0(t,\tau_*)$ taken to be $1$.
  According to the next lemma, we can use this simpler function $f_n$ as a stand-in for
  $r_n$.
  \begin{lemma}\label{lem:f_n.r_n}
    For some constant $1\leq C<\infty$ depending on $\vec{\nu}_*$, 
    it holds for all $n\geq 0$ and all $(t,\tau_*)\in\Ts$ that
    \begin{align*}
      \frac{1}{C} f_n(t,\tau_*) &\leq r_n(t,\tau_*)\leq Cf_n(t,\tau_*).
    \end{align*}
  \end{lemma}
  \begin{proof}
    We aim to show that for some $1\leq C<\infty$, it holds for all $n\geq 0$,
    $\vec{u}\in\Ww_n(t,\tau_*)$, and  $(t,\tau_*)\in\Ts$ with $\tau_*(R_t)=\os$ that
    \begin{align*}
      \frac{1}{C}\RST^{\os}[t,\tau_*]_n\leq \RST^{\zs}[t,\tau_*^{\vec{u}}]_n\leq C\,\RST^{\os}[t,\tau_*]_n,
      \intertext{and it holds for all $n\geq 0$,
    $\vec{u}\in\Ww_n(t,\tau_*)$, and  $(t,\tau_*)\in\Ts$ with $\tau_*(R_t)=\zs$ that}
      \frac{1}{C}\RST^{\zs}[t,\tau_*]_n\leq \RST^{\os}[t,\tau_*^{\vec{u}}]_n\leq C\,\RST^{\zs}[t,\tau_*]_n.
    \end{align*}
    Once we prove this, the result follows immediately from Lemma~\ref{lem:RST.rep}
    and the definition of $r_n$.
    
    To prove these statements, we go back to the unconditioned tree $(T,\omega_*)$.
    Let 
    \begin{align*}
      C_1=\frac{\max_{\sigma\in\{\zs,\os\}}\vec{\nu}_*(\sigma)}
             {\min_{\sigma\in\{\zs,\os\}}\vec{\nu}_*(\sigma)}.
    \end{align*}
    By Proposition~\ref{prop:(T,omega*)}\ref{i:star1},
    \begin{align*}
      \P\Bigl[(T,\omega_*)\in[t,\tau_*^{\vec{u}}]_n\Bigr]
        &= \P[T|_n=t|_n] \prod_{x\in L_n(t)} \vec{\nu}_*(\tau_*^{\vec{u}}(x))\\
        &\leq C_1\P[T|_n=t|_n] \prod_{x\in L_n(t)} \vec{\nu}_*(\tau_*(x))
        = C_1\P\Bigl[(T,\omega_*)\in[t,\tau_*]_n\Bigr]
    \end{align*}
    since $\tau_*^{\vec{u}}$ and $\tau_*$ match each other on $L_n(t)$ except
    at a single vertex, where one assigns $\zs$ and the other assigns $\os$.
    
    Suppose that $\tau_*(R_t)=\os$. Then
    \begin{align*}
      \RST^\zs[t,\tau_*^{\vec{u}}]_n &= \frac{\P\bigl[(T,\omega_*)\in[t,\tau_*^{\vec{u}}]_n\bigr]}{\vec{\nu}_*(\zs)}\\
        &\leq \frac{C_1\P\Bigl[(T,\omega_*)\in[t,\tau_*]_n\Bigr]}{\vec{\nu}_*(\zs)}\\
        &\leq \frac{C_1^2\P\Bigl[(T,\omega_*)\in[t,\tau_*]_n\Bigr]}{\vec{\nu}_*(\os)}
          = C_1^2\,\RST^{\os}[t,\tau_*]_n.
    \end{align*}
    The lower bound on $\RST^\zs[t,\tau_*^{\vec{u}}]_n$ and the bounds on $\RST^\os[t,\tau_*^{\vec{u}}]_n$
    follow by nearly identical proofs.
  \end{proof}
  
  Next, we recast $f_n(t,\tau_*)$ as a weighted sum over paths.
  \begin{defines}[$N(x,t,\tau_*)$ and $w_{t,\tau_*}(x,y)$]\label{def:N,w}
  For a vertex $x\in V(t)$, let $N(x,t,\tau_*)$ be the number of pivotal s-labeled
  children of $x$ in $(t,\tau_*)$. Suppose that $x$ and $y$ are respectively a vertex and its child
  in some path in $\Ww(t,\tau_*)$. Define
  \begin{align*}
    w_{t,\tau_*}(x,y)= N(x,t,\tau_*^{y})^{-1},
  \end{align*}
  which we will view as a weight on the edge from $x$ to $y$. We will shorten this to
  $w(x,y)$ when the tree $(t,\tau_*)$ is clear from context.
  \end{defines}
  In words, $w(x,y)$ is the reciprocal of the number of pivotal s-labeled children
  of $x$ \emph{after} swapping all labels on the path from the root to $x$ to $y$.
  Note that this count of pivotal $s$-labeled children is
  never zero for such an $x$ and $y$, since $y$ is always s-labeled
  and pivotal for $(t,\tau_*^y)$ as a consequence of belonging to a path in $\Ww(t,\tau_*)$.    
  \begin{lemma}\label{lem:f_n}
    For any $(t,\tau_*)\in\Ts$,
    \begin{align}
      f_n(t,\tau_*) &= \sum_{\vec{u}\in\Ww_n(t,\tau_*)} \prod_{i=0}^{n-1} w(u_i,u_{i+1}),\label{eq:weights}
    \end{align}
    which we can express recursively as
    \begin{align}
      f_n(t,\tau_*) &= \sum_x w(R_t,x) f_{n-1}(t(x),\tau_*|_{t(x)}),\label{eq:recursive}
    \end{align}
    where $x$ ranges over the pivotal s-labeled children of $R_t$.
  \end{lemma}
  \begin{proof}
    To prove \eqref{eq:weights}, we need to show that for any
    $\vec{u}=(u_0,\ldots,u_n)\in\Ww_n(t,\tau_*)$,
    \begin{align*}
      P_{\Ww_n(t,\tau_*^{\vec{u}})}(\vec{u}) &= \prod_{i=0}^{n-1}N\bigl(u_i,t,\tau_*^{u_{i+1}}\bigr)^{-1}.
    \end{align*}
    This is evident, as $P_{\Ww_n(t,\tau_*^{\vec{u}})}(\vec{u})$ is the probability
    that $\vec{u}$ is selected by the procedure of starting at the root in $(t,\tau_*^{\vec{u}})$ and
    uniformly picking a pivotal $s$-labeled child, then another pivotal s-labeled child, and so on.    
    Equation~\eqref{eq:recursive} follows from \eqref{eq:weights} by partitioning $\Ww_n(t,\tau_*)$
    into paths going through each of the s-labeled children of the root.
  \end{proof}
  Recall that $\omega_*^{\zs\to\os}$ is formed by swapping labels in
  $\omega_*^\zs$ along a random path $\vec{V}=(V_i)_{i\geq 0}$, where $V_0=R_{T^\zs}$.
  Call this path the \emph{spine} of $(T^\zs,\omega^{\zs\to\os})$.
  We now give some terminology for describing the weights (in the sense
  of Definitions~\ref{def:N,w}) of edges along and hanging off the spine.
  
  \begin{defines}[Definitions of $V_{i,j}$, $W_i$, $W_{i,j}$, $\Tspine$, $\Tspine_n$,
    $\Ggg$, and $\Ggg_n$] \label{def:spine}
  Let $V_{i,1},\ldots,V_{i,k_i}$ be the pivotal s-labeled children of $V_i$ 
  in $(T^\zs,\omega_*^{\zs\to\os})$ other than $V_{i+1}$.
  Let $W_i=w_{T^\zs,\omega_*^{\zs\to\os}}(V_i,V_{i+1})$, and let
  $W_{i,j}=w_{T^\zs,\omega_*^{\zs\to\os}}(V_i,V_{i,j})$.
  
  Let $\Tspine\subseteq T^\zs$ be the subtree consisting of $\vec{V}$ and all
  vertices $V_{i,j}$. Let $\Tspine_n$ be the restriction of $\Tspine$ to height~$n$.
  Let $\Ggg$ be the $\sigma$-algebra generated by $\vec{V}$, 
  $\Tspine$ and $\omega_*^{\zs\to\os}|_{\Tspine}$. Let  
  $\Ggg_n$ be the $\sigma$-algebra generated by $(V_0,\ldots,V_n)$, 
  $\Tspine_n$ and by $\omega^{\zs\to\os}|_{\Tspine_n}$.
  See Figure~\ref{fig:Tspine} for a depiction of the information
  captured by these $\sigma$-algebras. 
  \end{defines}
  The key idea in analyzing $f_n(T^\zs,\omega_*^{\zs\to\os})$ is that while $(T^\zs,\omega_*^{\zs\to\os})$
  behaves unusually along the spine, starting from any vertex $V_{i,j}$ it is
  a multitype Galton--Watson tree with the same child distributions as $(T,\omega_*)$.
  This is formally expressed in Proposition~\ref{prop:immigration}.
  Thus, understanding the weights along the spine of $(T^\zs,\omega_*^{\zs\to\os})$ as well
  as the weights on $(T,\omega_*)$ is enough to understand the weights on all
  of $(T^\zs,\omega_*^{\zs\to\os})$.  
  This is the same idea used by Pemantle, Peres, and Lyons to prove the Kesten--Stigum
  theorem (see \cite[Section 3]{LPP} or \cite[Chapter 12]{LP}).
  
    
  \begin{figure}
    \begin{center}
      \begin{tikzpicture}[yscale=1.41]
        \path (0,0)    node (V0) {\textbf{1s}}
              (-2,-1)  node (V1) {\textbf{1s}}
              (.5,-1)  node (V01) {\textbf{0s}}
              (1.5,-1) node (V02) {\textbf{0s}}
              (-2.75,-2) node (V2) {\textbf{0s}}
              (-1.75,-2) node (V11) {\textbf{0s}}
              (-0.75,-2) node (V12) {\textbf{0s}}
              (.25,-2) node (V13) {\textbf{1s}}
              (-3.4,-3) node (V3) {\textbf{1s}}
              (-1.9,-3) node (V21) {0s}
              (-0.4,-3) node (V22) {\textbf{1s}}
              ;
         \draw[thick] (V0)-- node[shift={(-.3,.1)}] {$\frac13$} (V1)
                          --node[shift={(-.3,-.1)}] {$\frac12$}(V2)
                          --node[shift={(-.25,-.1)}] {$\frac13$} (V3)
               (V0)--node[shift={(-.2,.1)}] {$\frac12$} (V01) (V0)--node[shift={(.25,.1)}] {$\frac12$} (V02)
               (V1)--node[shift={(-.15,-.1)}] {$\frac12$} (V11) 
               (V1)--node[shift={(-.15,-.1)}] {$\frac12$} (V12) 
               (V1)--node[shift={(.55,-.1)}] {$\frac14$} (V13) (V2)--node[shift={(.6,-.1)}] {$\frac13$} (V22);
         \draw (V2)--(V21) ;
      \end{tikzpicture}
    \end{center}
    
    \caption{A subset of $(T^\zs,\omega^{\zs\to\os})$. The tree automaton in this example
    assigns $1$ to a vertex if and only if it has exactly one $1$-labeled child.
    The spine, $(V_i)_{i\geq 0}$, is on the left
    side, and vertices $V_{i,1},\ldots,V_{i,k_i}$ hang off each $V_i$.  Pivotal vertices and the edges
    between them are in bold. Alongside each edge is its weight. For example, the weight from $V_0$
    to $V_1$ is $\frac13$, because if $V_0$ and $V_1$ have their labels swapped from $\os$ to $\zs$, then
    $V_0$ has three pivotal s-labeled children.}
    \label{fig:Tspine}
  \end{figure}
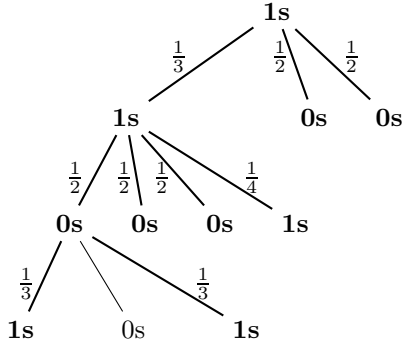

  \begin{prop}\label{prop:immigration}\ 
    \begin{enumerate}[(a)]
      \item The random variables $\bigl\{W_i,W_{i,j}\bigr\}_{i < n,\, 1\leq j\leq k_i}$ are measurable
        with respect to $\Ggg_n$. \label{item:imm.measurable}
      \item Conditional on $\Ggg$, the subtrees 
        \begin{align*}
          \Bigl\{\bigl(T^\zs(x),\omega^{\zs\to\os}_*|_{T^\zs(x)}\bigr)\colon
           \text{$x = V_{i,j}$ for some $i\geq 0$, $0\leq j\leq k_i$}\Bigr\}
        \end{align*}
          are independent. \label{item:imm.ind}
      \item For any $x=V_{i,j}$, the subtree $(T^\zs(x),\omega^{\zs\to\os}_*|_{T^\zs(x)})$
        is distributed conditional on $\Ggg$ as $\RST^\zs$ if $\omega_*^{\zs\to\os}(x)=\zs$ and as
        $\RST^\os$ if $\omega^{\zs\to\os}_*(x)=\os$.\label{item:imm.nonspine}
      \item The subtree $\bigl(T^\zs(V_n),\omega_*^{\zs\to\os}|_{T(V_n)}\bigr)$ is distributed
        conditional on $\Ggg_n$ as $\RST^{\zs\to\os}$ if $\omega_*^{\zs\to\os}(V_n)=\os$ and
        as $\RST^{\os\to\zs}$ if $\omega_*^{\zs\to\os}(V_n)=\zs$.\label{item:imm.spine}
    \end{enumerate}
  \end{prop}
  \begin{proof}
    Part~\ref{item:imm.measurable} follows directly from the definition.
    For parts~\ref{item:imm.ind} and \ref{item:imm.nonspine}, note
    that for any $x=V_{i,j}$, the subtree $(T^\zs(x),\omega^{\zs\to\os}_*|_{T^\zs(x)})$
    is identical to $(T^\zs(x),\omega^\zs_*|_{T^\zs(x)})$. We can also recharacterize
    $\Ggg$ as the $\sigma$-algebra generated
    by $\vec{V}$, $\Tspine$, and $\omega^\zs_*|_{\Tspine}$. Conditioning on $\Ggg$
    is then just revealing part of
    $(T^\zs,\omega_*^\zs)$, which is Galton--Watson by Proposition~\ref{prop:(T,omega*)}\ref{i:star3}. 
    Under this conditioning, the subtrees 
    \begin{align*}
      \bigl\{\bigl(T^\zs(x),\omega^{\zs}_*|_{T^\zs(x)}\bigr)\colon
    \text{$x = V_{i,j}$ for some $i\geq 0$, $0\leq j\leq k_i$}\bigr\}
    \end{align*}
    are unrevealed
    except for the labels of their roots, and hence they evolve independently according
    to $\RST^\zs$ or $\RST^\os$.
    
    To prove part~\ref{item:imm.spine}, we observe that conditioning on $\Ggg_n$ reveals
    $(V_0,\ldots,V_n)$, and it reveals a portion of $(T^\zs,\omega_*^\zs)$ with $V_n$ as a leaf.
    Thus $(T^\zs(V_n),\omega^\zs_*|_{T^\zs(V_n)})$ evolves either as $\RST^\zs$ or as $\RST^\os$ conditional
    on $\Ggg_n$, depending on $\omega^{\zs}_*(V_n)$. 
    Also, by its definition, $(V_i)_{i\geq n}$ conditional on $\Ggg_n$ is
    distributed as $P_{\Ww(T^\zs(V_n),\omega^\zs|_{T^\zs(V_n)})}$. Thus, $(T^\zs(V_n),\omega^{\zs\to\os}_*|_{T^\zs(V_n)})$
    is distributed conditionally on $\Ggg_n$ as stated.
  \end{proof}

  Now, we can start evaluating $\limsup_n f_n(T^\zs,\omega^{\zs\to\os}_*)$.
  First, we expand $f_n(T^\zs,\omega^{\zs\to\os}_*)$ in terms of the weights
  along and off the spine.
  \begin{lemma} \label{lem:f_n.expansion}
    We can express $f_n(T^\zs,\omega^{\zs\to\os}_*)$ as
    \begin{align}
      \begin{split}
            f_n(T^\zs,\omega^{\zs\to\os}_*) 
                &= \sum_{i=0}^{n-1} (W_0\cdots W_{i-1})
                    \sum_{j=1}^{k_i}W_{i,j}f_{n-i-1}\bigl(T^\zs(V_{i,j}),\omega^{\zs\to\os}_*|_{T^\zs(V_{i,j})}\bigr)\\
                     &\qquad\qquad+ W_0\cdots W_{n-1}
      \end{split}\nonumber\\
      \begin{split}
              &\leq \sum_{i=0}^{n-1} (W_0\cdots W_{i-1})
                    \sum_{j=1}^{k_i}W_{i,j}Cr_{n-i-1}\bigl(T^\zs(V_{i,j}),\omega^{\zs\to\os}_*|_{T^\zs(V_{i,j})}\bigr)
                      \\
                     &\qquad\qquad+ W_0\cdots W_{n-1},
      \end{split}\label{eq:f_n.expansion}
    \end{align}
    where $C$ is the constant from Lemma~\ref{lem:f_n.r_n}.
  \end{lemma}
  \begin{proof}
    The equality holds by successively applying \eqref{eq:recursive} from Lemma~\ref{lem:f_n},
    and the inequality is an application of Lemma~\ref{lem:f_n.r_n}.
  \end{proof}
  It is odd that we have bounded $f_n$ by $r_n$ when $f_n$ is a simpler quantity
  that we typically prefer to work with. But in the proof of Lemma~\ref{lem:X_n.finite},
  it will be easier to work with the Radon--Nikodym derivative itself rather than an approximation.
  
  Now, in Lemmas~\ref{lem:Yn.martingale}--\ref{lem:Yn.bounded},
  we prove some technical facts that help us bound \eqref{eq:f_n.expansion}.
  \begin{lemma}\label{lem:Yn.martingale}
    For any $i$ and $1\leq j\leq k_i$, the process
    \begin{align*}
      \Bigl(r_{n-i-1}\bigl(T^\zs(V_{i,j}),\omega^{\zs\to\os}_*|_{T^\zs(V_{i,j})}\bigr)\Bigr)_{n\geq i+1}
    \end{align*}
    conditional on $\Ggg$ is a nonnegative martingale in $n$ with mean one.
  \end{lemma}
  \begin{proof}
    By definition of $r_n$, given in \eqref{eq:r_n.def.1}--\eqref{eq:r_n.def.2}, the process
    is nonnegative.
    By Proposition~\ref{prop:immigration}\ref{item:imm.nonspine},
    the conditional distribution of
    \begin{align*}
      \Bigl(T^\zs(V_{i,j}),\;\omega^{\zs\to\os}_*|_{T^\zs(V_{i,j})}\Bigr)
    \end{align*}
    given $\Ggg$ is either $\RST^\zs$ or $\RST^\os$, depending on the value of $\omega^{\zs\to\os}_*(V_{i,j})$.
    For the sake of concreteness, suppose that $\omega^{\zs\to\os}_*(V_{i,j})=\zs$ so that
    its conditional distribution is $\RST^\zs$.
    Then $r_{n-i-1}\bigl(T^\zs(V_{i,j}),\omega^{\zs\to\os}_*|_{T^\zs(V_{i,j})}\bigr)$ conditional on $\Ggg$ is
    the Radon-Nikodym derivative of
    $\RST^{\os\to\zs}|_{\Gg_n}$ with respect to $\RST^{\zs}|_{\Gg_n}$, applied to an $\RST^\zs$-distributed
    random variable. Hence, conditional on $\Ggg$,
    it is a martingale in $n$ \cite[Lemma~5.3.4]{Durrett}. The same logic 
    shows that $r_{n-i-1}\bigl(T^\zs(V_{i,j}),\omega^{\zs\to\os}_*|_{T^\zs(V_{i,j})}\bigr)$ is a martingale
    conditional on $\Ggg$ when $\omega^{\zs\to\os}_*(V_{i,j})=\os$. The initial value of
    either martingale, when $n=i+1$, is $1$.
  \end{proof}

  \begin{lemma}\label{lem:two.step}
    For some $c<1$, it holds for all $n\geq 0$ that
    \begin{align*}
      \E[ W_nW_{n+1}\mid\Ggg_n] \leq c \text{ a.s.}
    \end{align*}
  \end{lemma}
  \begin{proof}
    First, we claim that $W_n$ is the reciprocal of the number of pivotal s-labeled children
    of $V_n$ in the original unswitched tree $(T^\zs,\omega_*^\zs)$; that is,
    \begin{align}\label{eq:W_n.expression}
      W_n=N(V_n,T^\zs,\omega_*^\zs)^{-1}.
    \end{align}
    Indeed, by the definition of $W_n$ in Definitions~\ref{def:spine} and then the definition
    of $w_{t,\tau_*}(x,y)$ in Definitions~\ref{def:N,w},
    \begin{align*}
      W_n &=w_{T^\zs,\omega_*^{\zs\to\os}}(V_i,V_{i+1})
          = N\bigl(V_n,T^\zs,(\omega_*^{\zs\to\os})^{V_{n+1}}\bigr)^{-1}.
    \end{align*}
    The labels $(\omega_*^{\zs\to\os})^{V_{n+1}}$ consist of the original labels $\omega_*^\zs$
    switched along the spine and then switched back again, yielding \eqref{eq:W_n.expression}.

    We now seek to analyze this expression conditional on $\Ggg_n$.
    By Proposition~\ref{prop:immigration}\ref{item:imm.spine},
    the distribution of $\bigl(T^\zs(V_n),\omega_*^{\zs\to\os}|_{T(V_n)}\bigr)$ conditional
    on $\Ggg_n$ is either $\RST^{\os\to\zs}$ or $\RST^{\zs\to\os}$, depending on
    $\omega_*^{\zs}(V_n)$. Equivalently, the distribution of
    $\bigl(T^\zs(V_n),\omega_*^\zs|_{T(V_n)}\bigr)$ conditional on $\Ggg_n$ is
    $\RST^\ell$ where $\ell=\omega_*^\zs(V_n)$, which is measurable with respect to $\Ggg_n$.
    Hence, $N(V_n,T^\zs,\omega_*^\zs)$ conditional on $\Ggg_n$ is distributed
    as $N(R_{T^\ell},T^\ell,\omega_*^\ell)$, where $\ell=\omega_*^\zs(V_n)$.

    Let $c_\ell = \E N(R_{T^\ell},T^\ell,\omega_*^\ell)^{-1}$ for $\ell=\zs,\os$.
    Recall that $N(R_{T^\ell},T^\ell,\omega_*^\ell)\geq 1$, since a pivotal s-labeled vertex must
    give birth to another pivotal s-labeled vertex.
    Hence, we have $c_\ell=1$ if and only if
    $N(R_{T^\ell},T^\ell,\omega_*^\ell)=1$ a.s.
    If $c_\zs<1$ and $c_\os<1$, then set $c=\max(c_\zs,c_\os)$ and use the easy bound $W_{n+1}\leq 1$
    to get
    \begin{align*}
      \E [W_nW_{n+1}\mid\Ggg_n]\leq c \text{ a.s.}
    \end{align*}
    
    The troublesome case is when $c_\zs=1$ or $c_\os=1$.
    Suppose $c_\zs=1$. If $c_\os=1$, the proof is identical
    with the roles of $\zs$ and $\os$ switched.
    The argument has two steps:
    \begin{enumerate}[(a)]
      \item $c_\os<1$; \label{i:two.step.1}
      \item a vertex of type $\zs$ gives birth to a pivotal $\os$-labeled vertex with positive probability.
        \label{i:two.step.2}
    \end{enumerate}
    Suppose that \ref{i:two.step.1} is false. Then $c_\zs=c_\os=1$, and consequently
    no vertex type in $(T,\omega_*)$
    ever gives birth to more than one $s$-labeled pivotal child. This implies that 
    no vertex gives birth to more than one pivotal child. Indeed, according
    to Proposition~\ref{prop:(T,omega*)}\ref{i:star2}, given the colours
    of the children of a vertex, their s- and d-labels are assigned independently. Thus, if a vertex
    of colour $0$ or $1$ had positive probability of having multiple pivotal children, 
    it would also have positive
    probability of having multiple s-labeled pivotal children.
    Since vertices of either colour give birth to at most one pivotal child,
    the highest eigenvalue of the matrix of means
    of $(\Tpiv,\omega)$ is at most one. But this contradicts our assumption throughout this section that
    the pivot tree is supercritical, establishing \ref{i:two.step.1}.
    
    For \ref{i:two.step.2}, from $c_\zs=1$ we know that a vertex of type~$\zs$
    always gives birth to exactly one pivotal s-labeled child. 
    As above, this implies that it always gives birth to exactly one pivotal child.
    In fact, a vertex of type~$0$ must always give birth to zero or one pivotal children,
    since if it had positive probability of giving birth to two or more, then it would
    have positive probability of giving birth to two or more pivotal s-labeled vertices, 
    and so a vertex of type~$\zs$ would have positive probability of
    giving birth to two or more pivotal children. Hence $m_{00}+m_{01}\leq 1$,
    in the language of Lemma~\ref{lem:growth.rates}.
    
    Suppose that \ref{i:two.step.2} is false and a vertex of type~$\zs$ never gives
    birth to a pivotal $\os$-labeled vertex. Then, a vertex of type~$0$ never gives birth
    to a pivotal vertex of type~$1$ (if it had positive probability of doing so,
    then a vertex of type~$\zs$ would have positive probability of giving birth
    to a pivotal vertex of type~$\os$). Hence, $m_{01}=0$. By Lemma~\ref{lem:growth.rates},
    the matrix of means for $\Tpiv$ has the form $\bigl[\begin{smallmatrix}m_{00}&0\\0&m_{00}
    \end{smallmatrix}\bigr]$ where $m_{00}\leq 1$. But this contradicts the supercriticality
    of $\Tpiv$, proving \ref{i:two.step.2}.
    
    Now, we are ready to evaluate $\E [W_nW_{n+1}\mid\Ggg_n]$ when $c_\zs=1$.
    Let $p$ be the probability that a vertex of type~$\zs$ has a pivotal child of type~$\os$,
    which we know to be positive by \ref{i:two.step.2}.
    If $\omega_*^{\zs}(V_n)=\zs$, then
    \begin{align*}
      \E [W_nW_{n+1}\mid\Ggg_n] &= 1-p + pc_\os<1.
    \end{align*}
    If $\omega_*^{\zs}(V_n)=\os$, then using the bound $W_{n+1}\leq 1$, we have
    \begin{align*}
      \E [W_nW_{n+1}\mid\Ggg_n] &\leq c_\os<1.
    \end{align*}
    We then take $c$ as the maximum of these two values to complete the proof.
  \end{proof}
  
  We mention that the difficult case in this proof, where one of $c_\zs$ and $c_\os$ is
  zero, can truly occur. 
  For example, let $A$ be the two-state automaton assigning
  colour~$1$ to a parent if and only if the parent has at least three children,
  all of which have the same colour. If a vertex has colour~$0$, then
  it has a pivotal child only when it has three or more children and all but one of them
  are the same colour, in which case it has exactly one pivotal child (the odd-coloured one).
  Thus, conditional on being type~$\zs$, a vertex has exactly one pivotal s-labeled 
  child, which makes $c_\zs=1$.
  
  The following lemma is well known, though
  it is most commonly stated with its converse under the additional
  assumption that $X_1,X_2,\ldots$ are independent (see \cite[Exercise~12.2]{LP}).
  We sketch the proof here.
  \begin{lemma}\label{lem:log.moment.bound}
    Let $X_1,\,X_2,\ldots$ be nonnegative random variables with a common distribution. 
    If this distribution has finite log-moment, then
    \begin{align*}
      \sum_{n=1}^{\infty} c^n X_n < \infty \text{ a.s.}
    \end{align*}
    for all $c\in(0,1)$.
  \end{lemma}
  \begin{proof}
    Apply the Borel--Cantelli lemma to show that $\limsup_{n\to\infty}\frac1n\log X_n=0$
    a.s. Then it follows for some finite random $N$ that $X_n\leq (2c)^{-n}$ for $n\geq N$.
  \end{proof}

  \begin{lemma}\label{lem:Yn.bounded}
    If the child distribution $\chi$ has finite log-moment, then
    \begin{align*}
       \sum_{i=0}^{\infty}W_0\cdots W_{i-1}
                 N\bigl(V_i, T^\zs, \omega^{\zs\to\os}_*\bigr) < \infty \text{ a.s.}
    \end{align*}
  \end{lemma}
  \begin{proof}
    First, we claim that
    \begin{align*}
      \E[W_0\cdots W_{i-1}] &\leq c^{\lfloor i/2 \rfloor}
    \end{align*}
    for some $c<1$.
    This is proven by applying
    Lemma~\ref{lem:two.step} to take conditional expectations given $\Ggg_{i-2}$,
    then given $\Ggg_{i-4}$, and so on. Choosing any $b\in(\sqrt{c},1)$ and applying Markov's inequality,
    \begin{align*}
      \sum_{i=1}^{\infty}\P[W_0\cdots W_{i-1} > b^i] \leq \sum_{i=1}^{\infty}\frac{c^{\lfloor i/2 \rfloor}}{b^i}
        <\infty.
    \end{align*}
    By the Borel--Cantelli lemma, it holds almost surely that $W_0\cdots W_{i-1}\leq b^i$
    for all but finitely many values of $i$. Hence, it suffices to show that
    \begin{align}\label{eq:Y_n.modified.bound}
      \sum_{i=0}^{\infty} b^iN(V_i,T^\zs,\omega^{\zs\to\os}_*)<\infty \text{ a.s.}
    \end{align}

    Now, it remains to apply Lemma~\ref{lem:log.moment.bound}.
    Since $N(V_i,T^\zs,\omega^{\zs\to\os}_*)$ is the number of s-pivotal offspring of $V_i$
    in $(T^\zs,\omega^{\zs\to\os}_*)$, it is bounded by the total number of offspring of $V_i$
    in $T^\zs$. 
        By Proposition~\ref{prop:immigration}\ref{item:imm.nonspine}, the distribution
    of $\bigl(T(V_i),\omega_*^{\zs\to\os}|_{T(V_i)}\bigr)$ conditional on $\Ggg$ is either
    $\RST^\zs$ or $\RST^\os$. 
    Thus, conditional on $\Ggg$, the random variable $N(V_i,T^\zs,\omega^{\zs\to\os}_*)$
    is stochastically dominated by the number of vertices at level~$1$ of
    either $T^\zs$ or $T^\os$. Let $A_\zs$ and $A_\os$ be random variables with these distributions,
    respectively.
    Since $\chi$ is assumed to have finite log-moment,
    both $A_\zs$ and $A_\os$ have finite log-moment as well.    
    Now, let $X_i$ have any distribution that stochastically dominates $A_\zs$
    and $A_\os$ and has finite log-moment. For example, one could take $X_i=A_\zs+A_\os$
    where $A_\zs$ and $A_\os$ are independent. Thus, $N(V_i,T^\zs,\omega^{\zs\to\os}_*)$
    conditional on $\Ggg$ is stochastically dominated by $X_i$, and so there
    exists a coupling in which $N(V_i,T^\zs,\omega^{\zs\to\os}_*)\leq X_i$ for all $i$. 
    (Note that we 
    do not care about the joint distribution of $X_1,X_2,\ldots$.) Under this coupling,
    \begin{align*}
      \sum_{i=0}^{\infty} b^iN(V_i,T^\zs,\omega^{\zs\to\os}_*)\leq \sum_{i=0}^{\infty} b^iX_i,
    \end{align*}
    which is almost surely finite by Lemma~\ref{lem:log.moment.bound}.
    This proves \eqref{eq:Y_n.modified.bound}, from which the lemma follows.
  \end{proof}
  
  Finally, we are ready to achieve what we have been building towards by bounding 
  the right-hand side of \eqref{eq:f_n.expansion}.

  \begin{lemma}\label{lem:X_n.finite}
    Assume that the child distribution $\chi$ has finite log-moment.
    Then
    \begin{align*}
      \limsup_{n\to\infty} f_n(T^\zs,\omega^{\zs\to\os}_*)<\infty \text{ a.s.}
    \end{align*}
  \end{lemma}
  \begin{proof}
    Let 
    \begin{align*}
      Y_n = \sum_{i=0}^{n-1} (W_0\cdots W_{i-1})
                    \sum_{j=1}^{k_i}W_{i,j}Cr_{n-i-1}\bigl(T^\zs(V_{i,j}),\omega^{\zs\to\os}_*|_{T^\zs(V_{i,j})}\bigr),
    \end{align*}
    one of the terms on the right-hand side of \eqref{eq:f_n.expansion}.
    We will show that $Y_n$ converges almost surely to a finite value.
    The idea is to use Lemmas~\ref{lem:Yn.martingale} and \ref{lem:Yn.bounded}
    to show that the conditional distribution
    of $(Y_n)_{n\geq 0}$ given $\Ggg$ is that of a submartingale bounded in $L^1$.
            
    First, consider $Y_n$ conditionally on $\Ggg$. 
    By Proposition~\ref{prop:immigration}\ref{item:imm.measurable},
    the random variables $W_0,W_1,\ldots$ are constants.
    By Proposition~\ref{prop:immigration}\ref{item:imm.nonspine},
    the processes
    \begin{align*}
      \Bigl(r_{n-i-1}\bigl(T^\zs(V_{i,j}),\omega^{\zs\to\os}_*|_{T^\zs(V_{i,j})}\bigr)\Bigr)_{n\geq i+1}
    \end{align*}
    are independent for different values of $i$.
    By Lemma~\ref{lem:Yn.martingale}, these processes are martingales.
    Hence, $Y_n$ conditional on $\Ggg$ is a sum of independent martingales with an additional martingale
    added at each step, which makes it a submartingale conditional on $\Ggg$.

    To prove that $\sup_n\E[Y_n\mid \Ggg]<\infty$ a.s.,
    by Lemma~\ref{lem:Yn.martingale},
    \begin{align*}
      \E\Bigl[r_{n-i-1}\bigl(T^\zs(V_{i,j}),\omega^{\zs\to\os}_*|_{T^\zs(V_{i,j})}\bigr)\Bigmid\Ggg\Bigr]=1 
         \text{ a.s.}
    \end{align*}
    Hence,
    \begin{align}
      \sup_n\E[ Y_n \mid \Ggg ] &= \sup_n\sum_{i=0}^{n-1} (W_0\cdots W_{i-1})\sum_{j=1}^{k_i} W_{i,j}\nonumber\\
        &\leq \sup_n\sum_{i=0}^{n-1} W_0\cdots W_{i-1}N\bigl(V_i, T^\zs, \omega^{\zs\to\os}_*\bigr)\nonumber\\
        &\leq \sum_{i=0}^{\infty}W_0\cdots W_{i-1}N\bigl(V_i, T^\zs, \omega^{\zs\to\os}_*\bigr)<\infty \text{ a.s.}
        \label{eq:Y_n.bound}
    \end{align}
    The first inequality uses the bound $W_{i,j}\leq 1$ along with the fact
    that $k_i=N\bigl(V_i, T^\zs, \omega^{\zs\to\os}_*\bigr)-1$.
    The last inequality is the statement of Lemma~\ref{lem:Yn.bounded}.
        
    We have now shown that $(Y_n)_{n\geq 0}$ conditional on $\Ggg$ 
    is a submartingale bounded in $L^1$. It hence converges almost
    surely to a finite limit. Since $W_0\cdots W_{n-1}$ is a decreasing positive sequence in $n$,
    it also converges as $n\to\infty$. By Lemma~\ref{lem:f_n.expansion},
    we have shown that $f_n(T^\zs,\omega^{\zs\to\os}_*)$
    is bounded by a process converging almost surely to a finite limit as $n\to\infty$.
  \end{proof}

  \begin{proof}[Proof of Theorem~\ref{main 2 colours} ($\Longrightarrow$)]
    Suppose that $\Tpiv$ is supercritical.
    By Lemmas~\ref{lem:f_n.r_n} and \ref{lem:X_n.finite},
    \begin{align*}
      \limsup_{n\to\infty} r_n(T^\zs,\omega^{\zs\to\os}_*)<\infty \text{ a.s.}
    \end{align*}
    By Lemma~\ref{lem:dichotomy}, we have $\RST^{\zs\to\os}\ll\RST^\os$.
    Therefore any almost sure event under $\RST^\os$ is almost sure under $\RST^{\zs\to\os}$
    as well.
    
    Suppose that $\vec{\nu}$ is interpretable. By Proposition~\ref{prop:interpretability},
    the random variable $\omega(R_T)$ is measurable with respect to $T$. Hence there exists
    a measurable map $\varphi\colon \Tt\to\{0,1\}$ such that
    $\omega(R_T)=\varphi(T)$ a.s. Since $T^\zs$ and $T^\os$ are distributed as $T$ conditioned 
    on subevents of $\omega(R_T)=0$ and $\omega(R_T)=1$, respectively, we have
    $\varphi(T^\zs)=0$ a.s.\ and $\varphi(T^\os)=1$ a.s.
    Stating the second of these facts in terms of measure theory, the event
    $\{(t,\tau_*)\in\Ts\colon\varphi(t)=1\}$ has probability one under $\RST^\os$.
    Hence it has probability one under $\RST^{\zs\to\os}$ as well. Since 
    $(T^\zs,\omega^{\zs\to\os}_*)\sim\RST^{\zs\to\os}$, this shows that $\varphi(T^\zs)=1$ a.s.,
    a contradiction.
  \end{proof}

\begin{remark}\label{rmk:kstate.supercritical}
  The main difficulty in extending this proof to the case $\abs{\Sigma}\geq 3$
  is that the regularity properties proven in Section~\ref{sec:pivot.regularity}
  for $\abs{\Sigma}=2$ do not necessarily hold when $\abs{\Sigma}=3$.
  For example, when $\abs{\Sigma}=2$, if the pivot tree is supercritical,
  then it survives with positive probability conditional on either $\omega(R_T)=0$ or $\omega(R_T)=1$
  by Proposition~\ref{prop:2state.regularity}\ref{i:supercrit}, which let us define
  measures $\RST^\zs$ and $\RST^\os$. When $\abs{\Sigma}\geq 3$, if the pivot tree
  is supercritical, it must survive with positive probability from some starting state,
  but it is not obvious that it must do so from multiple starting states.
  Nonetheless, we expect that it can be done and plan to address it in future work.
\end{remark}

\section{Applications of the main results}\label{sec:examples}

\subsection{Monotone tree automata}

\par We introduce here a special class of tree automata, which we term \emph{monotone tree automata}. This class encompasses tree automata that arise out of many naturally occurring EMSO properties of rooted trees. 

\par Consider an automaton $A$ with set of colours $\Sigma$. Suppose that $\Sigma$ has a total ordering on it, so that without loss of generality, we set
$\Sigma=\{0,\ldots s\}$. Let $\vec{n} = (n_{i}\colon 0 \leq i \leq s)$ and
$\vec{m} = (m_{i}\colon 0 \leq i \leq s)$, where $n_i$ and $m_i$ represent
counts of children of type $i$. If $\sum_{i=0}^{s} n_i=\sum_{i=0}^{s} m_i$, then
we write $\vec{n}\preceq\vec{m}$ if one can modify the configuration of children from $\vec{n}$
to become $\vec{m}$ by only increasing the colours of children. (For example, $(1,2,1)\preceq(1,1,2)$,
since one moves from children $0,1,1,2$ to $0,1,2,2$ by increasing the colour of a child from $1$ to $2$.)
The automaton $A$ is called \emph{monotone} if $\vec{n}\preceq\vec{m}$ implies that $A(\vec{n})\leq A(\vec{m})$.

In the following lemma we state a notable characteristic of the pivot tree when we consider a monotone tree automaton $A$ on two states.

\begin{lemma}\label{monotone_single_state_pivot}
Consider a tree automaton $A$ on two colours. Then $A$ is monotone if and only if pivotal
children always have the same colour as their parents.
\end{lemma}
\begin{proof}

Let $\Sigma=\{0,1\}$. Assume pivotal children always have the same state as their parents, and consider
two configurations of children $\vec{n}\preceq\vec{m}$. Suppose that $A(\vec{n})=1$.
We can move from $\vec{n}$ to $\vec{m}$ only by changing vertices from state~$0$ to $1$.
These vertices are never pivotal, so $A(\vec{m})=1$. Since $A(\vec{n})=1$ implies $A(\vec{m})=1$,
the automaton $A$ is monotone.

Conversely, suppose that a node can have a pivotal child of the opposite state of itself. Then
swapping this child's state changes the parent's state in the opposite direction, showing that $A$
is not monotone.
\end{proof}

\begin{lemma}\label{lem:2state.same.growth}
  Let $A$ be a monotone tree automaton with $\Sigma = \{0,1\}$,
  and let $\Tpiv$ be the pivot tree associated with some fixed point. Then
  \begin{align*}
    \E[\ell_1(\Tpiv) \mid \omega(R_T)=0] = \E[\ell_1(\Tpiv) \mid \omega(R_T)=1].
  \end{align*}
  That is, the expected number of pivotal children that a vertex has is the same regardless
  of whether the vertex is labeled $0$ or $1$.
\end{lemma}
\begin{proof}
  By Lemma~\ref{monotone_single_state_pivot},
  \begin{align*}
    \E[\ell_1(\Tpiv) \mid \omega(R_T)=0] &= \E[Z_0\mid\omega(R_T)=0],\\\intertext{and}
    \E[\ell_1(\Tpiv) \mid \omega(R_T)=1] &= \E[Z_1\mid\omega(R_T)=1],
  \end{align*}
  using the notation of Lemma~\ref{lem:growth.rates}. By this lemma,
  these quantities are equal.
\end{proof}

When $\abs{\Sigma}=2$, the automaton distribution map $\Psi$ maps
a distribution $\Ber(x)$ to $\Ber(y)$. We thus abuse notation and
treat $\Psi$ as a map from $[0,1]$ to itself, writing
$\Psi(x)=y$ instead of $\Psi(\Ber(x))=\Ber(y)$.
We also say that $p\in[0,1]$ is a fixed point of $\Psi$ rather than saying
that $\Ber(p)$ is.

In the next lemma, we give a convenient way of determining whether a fixed point corresponding to a given monotone automaton with two colours is rogue or not.

\begin{lemma}\label{growth_rate_derivative}
Suppose $A$ is a monotone automaton with set of colours $\Sigma = \{0, 1\}$. For a fixed point 
$\vec{\nu}=\Ber(p)$ with $0<p<1$, the growth rate of the pivot tree is equal to $\Psi'(p)$.
\end{lemma} 

\begin{proof}
In Lemma~\ref{prop:2state.regularity}, we show that if $M$ is the matrix of means for $\Tpiv$, then the spectral radius is equal to the expected number of pivotal children of the root, which is the growth rate of $\Tpiv$. Thus all we have to establish is that $\Psi'(p)$ is equal to the spectral radius of $M$.

The value of $\Psi(x)$ is given by the following procedure: Sample a number of children from $\chi$
and assign them i.i.d.\ $\Ber(x)$ states. Then, apply the automaton to determine the state
of the parent. Then $\Psi(x)$ is the expected value of the parent. Abusing notation
slightly by letting $A$ act on a sequence of states as we did in the proof
of Lemma~\ref{lem:growth.rates}, we have
\begin{align*}
  \Psi(x) = \E^x[A(X_1,\ldots,X_K)],
\end{align*}
where $K\sim\chi$ and $(X_i)_{i\geq 1}$ are i.i.d.\ $\Ber(x)$ under $\E^x$.
Let $P$ denote the number of coordinates of $(X_1,\ldots,X_k)$ that are pivotal for
$A$ at $(X_1,\ldots, X_k)$.
By the Margulis--Russo formula \cite[Theorem~3.2]{GS},
\begin{align*}
  \frac{d}{dx}\E^x[A(X_1,\ldots,X_k)\mid K] = \E^x[P\mid K].
\end{align*}
Taking expectations,
\begin{align*}
  \Psi'(x) = \E^x[P].
\end{align*}
Under $\E^p$, the random variable $P$ has the distribution of the number of pivotal
children of the root of $(T,\omega)$. Hence,
\begin{align*}
 \Psi'(p) &= \E[\ell_1(\Tpiv)].\qedhere
\end{align*}
\end{proof}

\begin{remark}
  Since a fixed point $p$ of $\Psi$ is attractive if $\abs{\Psi'(p)}< 1$,
  this lemma together with Theorem~\ref{main 2 colours} shows that for a monotone two-state automaton,
  an attractive fixed point is always interpretable. We mention that this is not
  true for nonmonotone automata. For example, the fixed point in
  Example~\ref{int_eg_2} can be computed to be rogue for $\lambda=4$ despite being attractive.
\end{remark}

We are finally ready to answer Question~\ref{q:spencer}.
Recall that the at-least-two automaton assigns the parent state~$1$ if and
only if at least two children have state~$1$.
We mentioned in the introduction that with Poisson child distribution,
this automaton has either one, two, or three fixed points depending
on $\lambda$. We will prove this in detail now, and we will classify the fixed points
as rogue or interpretable.
\begin{example}\label{ex:answer}
  Let $A$ be the at-least-two automaton, and let $\chi\sim\Poi(\lambda)$.
 As we saw
in \eqref{eq:at.least.two.recurrence}, the automaton distribution map
is
\begin{align*}
  \Psi(x) = 1 - e^{-\lambda x}(1+\lambda x).
\end{align*}
Define
\begin{align*}
  \lambdacrit &= \min_{x>0} \frac{x}{1-e^{-x}(1+x)}\approx 3.35.
\end{align*}
The function $x/\bigl(1-e^{-x}(1+x)\bigr)$ is convex on $(0,\infty)$ and hence has a unique minimizer,
which we denote by $x^*$. 
Now, substituting $\lambda x$ for $x$ in the function to be minimized, consider the inequality 
\begin{align}
  \frac{\lambda x}{\Psi(x)}\leq \lambda\label{eq:Psi.ineq}
\end{align}
on $(0,\infty)$.
If $\lambda<\lambdacrit$, it has no solutions, since $\lambda x/\Psi(x)\geq\lambdacrit$.
Hence $\Psi(x)<x$ for $x>0$, demonstrating that $\Psi$ has no fixed points other than
the trivial $x=0$.
If $\lambda=\lambdacrit$, then \eqref{eq:Psi.ineq} has exactly one solution.
The solution is $x=x^*/\lambdacrit$, and equality occurs in \eqref{eq:Psi.ineq}, making
it a fixed point of $\Psi$.
Since $x^*<\lambdacrit$, the solution lies in $(0,1)$. Hence $\Psi$ has one nontrivial fixed point in this case.
If $\lambda>\lambdacrit$, then \eqref{eq:Psi.ineq} has an interval of solutions $[a,b]$, which
contains $x^*/\lambda\in(0,1)$. It is easy to check that
$a>0$ and $b<1$. Thus $\Psi(x)$ lies below the line $y=x$ on $(0,a)$, then
sits above it on $(a,b)$, and then lies below it on $(b,1]$, giving $\Psi$ fixed points
$a,b$ in addition to $0$.

Now, assume that $\lambda>\lambdacrit$, so that $\Psi$ has fixed points $0$, $a$, and $b$.
Question~\ref{q:spencer} asks whether there exists a classification of trees
into states $\{0,1\}$ such that a tree has state~$1$ if and only if it has at least
two children of state~$1$, and a Galton--Watson tree with $\Poi(\lambda)$
child distribution has state~$1$
with probability $a$. In other words, the question is whether
$a$ is rogue or interpretable. By Lemma~\ref{growth_rate_derivative},
we can answer this question by finding $\Psi'(a)$. 
Since $\Psi(x)$ lies under the curve $y=x$ up until $x=a$ and then rises above it,
its derivative at $x=a$ exceeds $1$. Therefore $a$ is a rogue solution by
Lemma~\ref{growth_rate_derivative} and Theorem~\ref{main 2 colours}.
\end{example}

The following result uses a similar approach:
\begin{prop}
  The highest and lowest fixed points of a monotone two-state automaton are always interpretable.
\end{prop}
\begin{proof}
  If $\Psi(1)=1$, then the highest fixed point is $1$,
  which has the trivial interpretation $t\mapsto 1$. 
  Otherwise $\Psi(1)<1$, and at the highest fixed point the graph of $\Psi$ is either
  crossing from above the line $y=x$ to below, or it has $y=x$ as a tangent line (note that $\Psi$
  is continuously differentiable). In either
  case $\Psi'(x)\leq 1$ at the fixed point, making it interpretable by Lemma~\ref{growth_rate_derivative} 
  and Theorem~\ref{main 2 colours}.
  The same argument shows that the smallest fixed point is also interpretable.
\end{proof}

\subsection{More examples of two-state automata}\label{int_egs}
We give some examples of tree automata for which a fixed point undergoes a phase transition
from interpretable to rogue as a parameter of the child distribution is varied. 
We also give a numerical example to show that a two-state automaton can have
many fixed points.
As in the previous section, we write $\Psi(x)=y$ rather than $\Psi(\Ber(x))=\Ber(y)$
for two-state automata.
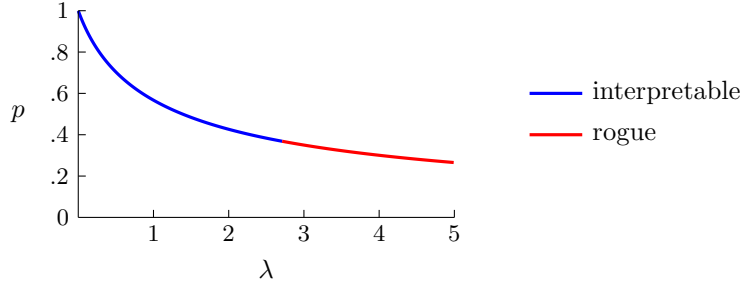
\begin{figure}
\centering
    \begin{tikzpicture}[yscale=2.75]
      \draw[very thick,blue] plot file{A1010green.table};
      \draw[very thick,red] plot file{A1010red.table}; 
      \draw (0,0) -- (0,1);
      \foreach \x in {.2, .4, .6, .8, 1}
        \draw (.1, \x)--(0,\x) node[left,font=\small] {$\x$};   
      \draw (0,0) node [left,font=\small] {$0$};
      \path (-.8, .5) node {$p$}
            (2.5, -.25) node {$\lambda$};
      \draw (0,0) -- (5,0);
      \foreach \x in {1, 2, 3, 4, 5}
        \draw (\x, .03636)--(\x, 0) node[below,node font=\small] {$\x$};                                     
      \draw[very thick,blue] (6,.6)--(6.7,.6) node[black,anchor=mid west] {interpretable};
      \draw[very thick,red] (6,.4)--(6.7,.4) node[black,anchor=mid west] {rogue};
    \end{tikzpicture}
\caption{Plot for Example~\ref{ex.A1010}, showing the fixed point for
  the automaton assigning the parent state `1' if and only if there are zero `1'~children
  with a $\Poi(\lambda)$ child distribution. 
}
\label{fig:A1010}
\end{figure}
\begin{example}\label{int_eg_3}\label{ex.A1010}
Consider the automaton where a node is in state~$1$ if and only if it has zero children in state~$1$,
given formally by the map $(n_0,n_1)\mapsto\1\{n_1=0\}$.
The distributional map corresponding to this automaton with $\Poi(\lambda)$ offspring distribution is
\begin{equation*}
\Psi(x) = e^{-\lambda x}.
\end{equation*}
Notice that if we consider the function $f_{\lambda}(x) = e^{-\lambda x} - x$, then $f''_{\lambda}(x) = \lambda^{2} e^{-\lambda x} > 0$, showing that it is a convex function. Moreover, $f_{\lambda}(0) = 1$ and $f_{\lambda}(1) = e^{-\lambda} - 1 < 0$ for all $\lambda > 0$. Hence $f_{\lambda}$ has a unique root in $(0,1)$, which tells us that there is a unique fixed point $x(\lambda)$ of $\Psi$.

For this automaton, a node in state~$0$ has pivotal children if and only if it has exactly one child in state~$1$ (whatever the number of $0$-state children may be), and this child will be pivotal. 
A node in state~$1$ has pivotal children if and only if it has at least one child, in which case every child will be pivotal. If $X_{0}$ denotes the total number of $0$-state children and $X_{1}$ the total number of $1$-state children of the root, then
\begin{align*}
  \E[ \ell_1(\Tpiv) ] &= \E\bigl[ \1\{X_1=1\} + X_0\1\{X_1=0\} \bigr]\\
    &= \lambda x e^{-\lambda x} + \E[ X_0] \P[X_1=0]\\
    &= \lambda x e^{-\lambda x} + \lambda(1-x)e^{-\lambda x} = \lambda e^{-\lambda x}.
\end{align*}
Thus, to determine if the fixed point is rogue or interpretable with Theorem~\ref{main 2 colours},
we have to determine if $\lambda e^{-\lambda x(\lambda)}\leq 1$.

Rewriting the equation $x(\lambda)=e^{-\lambda x(\lambda)}$, we find that
\begin{equation}\label{rewrite}
\log x(\lambda) = - \lambda x(\lambda) \quad \iff \quad y(\lambda) \log y(\lambda) = \lambda,
\end{equation}
where $y(\lambda) = \left(x(\lambda)\right)^{-1}$. Noting that the function $u \log u$ is strictly increasing over all $u > 1$ (which is the range we care about), we conclude that $\lambda > e$ if and only if $y(\lambda) > e$. In that case, from the first equation of \eqref{rewrite}, we have 
\begin{equation*}
\lambda e^{-\lambda x(\lambda)} = \lambda x(\lambda) = \log y(\lambda) > 1.
\end{equation*}
This shows that $E\left[\ell_{1}(\Tpiv)\right] > 1$ for $\lambda > e$, and $E\left[\ell_{1}(\Tpiv)\right] \leq 1$ for $\lambda \leq e$. 
By Theorem~\ref{main 2 colours} and Proposition~\ref{prop:2state.regularity}\ref{i:eig},
the fixed point is interpretable for $\lambda \leq e$ and rogue for $\lambda > e$. 
This is illustrated in Figure~\ref{fig:A1010}. 
\end{example}

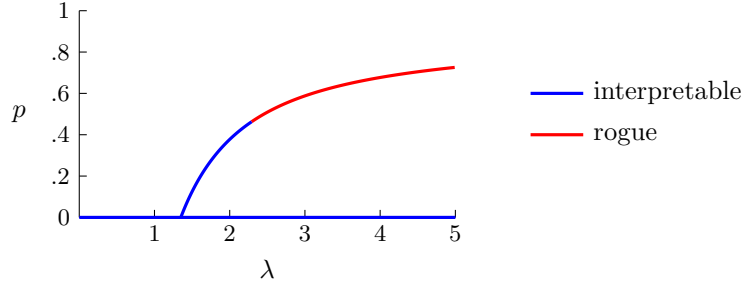
\begin{figure}
\centering
    \begin{tikzpicture}[yscale=2.75]
      \draw (0,0) -- (0,1);
      \foreach \x in {.2, .4, .6, .8, 1}
        \draw (.1, \x)--(0,\x) node[left,font=\small] {$\x$};
      \foreach \x in {1, 2, 3, 4, 5}
        \draw (\x, .03636)--(\x, 0) node[below,node font=\small] {$\x$};
      \draw (0,0) node [left,font=\small] {$0$};
      \path (-.8, .5) node {$p$}
            (2.5, -.25) node {$\lambda$};
      \draw (0,0) -- (5,0);
      \draw[very thick,red] plot file{A0001red.table};
      \draw[very thick,blue] plot file{A0001green.table};
      \draw[very thick,blue] (0,0)--(5,0);
      \draw[very thick,blue] (6,.6)--(6.7,.6) node[black,anchor=mid west] {interpretable};
      \draw[very thick,red] (6,.4)--(6.7,.4) node[black,anchor=mid west] {rogue};
    \end{tikzpicture}
\caption{Plot for Example~\ref{int_eg_2}, showing the fixed point for
  the automaton assigning the parent state `1' if and only if it has at least one `0' child
  and at least one `1' child
  with a $\Poi(\lambda)$ child distribution.}
  \label{fig:A0001}
\end{figure}

\begin{example}\label{int_eg_2}
Consider the automaton $A$ on colour set $\Sigma = \{0, 1\}$, where a node is in state $1$ if and only if it has at least one child in state $0$ and at least one child in state $1$. That is, the automaton is the map $(n_0,n_1)\mapsto\1\{(n_0\geq 1)\land(n_1\geq 1)\}$. The distributional map for this automaton
with child distribution $\Poi(\lambda)$ is given by
\begin{equation}\label{psi_one_0_one_1}
\Psi(x) = 1 - e^{-\lambda(1-x)} - e^{-\lambda x} + e^{-\lambda}.
\end{equation}
Note that 
\begin{align}
\Psi'(x) = -\lambda e^{-\lambda (1-x)} + \lambda e^{-\lambda x},
\end{align}
and
\begin{align}
\Psi''(x) = -\lambda^{2} e^{-\lambda (1-x)} - \lambda^{2} e^{-\lambda x} < 0,
\end{align}
showing that $\Psi$ is strictly concave. We observe that $\Psi(0)=0$ and
$\Psi'(0)= \lambda(1-e^{-\lambda})$. Let $\lambda_0\approx 1.35$ be the unique
solution to $\lambda(1-e^{-\lambda})=1$. If $\lambda\leq\lambda_0$, we have
$\Psi'(0)\leq 1$, and the graph of $\Psi(x)$ stays below the line $y=x$ for all $x>0$.
Thus $0$ is the only fixed point of $\Psi(x)$ in this case.
If $\lambda > \lambda_0$, then $\Psi'(0)>1$. Since $\Psi(1)=0$, this implies that the
graph of $\Psi(x)$ rises above the line $y=x$ and then dips back below it, giving rise
to a nontrivial fixed point we denote by $x(\lambda)$.

Let $X_i$ be the number of children of the root in state~$i$.
We summarize all configurations in which any of these children are pivotal:
\begin{description}
  \item[Root in state~0]\ 
    \begin{itemize}
      \item $X_0\geq 2$, $X_1=0$: $X_0$ pivotal children
      \item $X_0=0$, $X_1\geq 2$: $X_1$ pivotal children
    \end{itemize}
  \item[Root in state~1]\ 
    \begin{itemize}
      \item $X_0=X_1=1$: two pivotal children
      \item $X_0\geq 2$, $X_1=1$: one pivotal child
      \item $X_0=0$, $X_1\geq 2$: one pivotal child
    \end{itemize}
\end{description}
\noindent Thus,
\begin{align*}
  \E[\ell_1(\Tpiv)] &= \E\Bigl[ X_0\1\{X_0\geq 2,X_1=0\} + X_1\1\{X_0=1, X_1\geq 2\}
     \\&\qquad\quad\quad+\1\{X_0=X_1=1\} + \1\{X_0\geq 2, X_1=1\} + \1\{X_0=1,X_1\geq 2\} \Bigr]\\
   &= e^{-\lambda x}\lambda(1-x)\bigl(1-e^{-\lambda(1-x)}\bigr) +
      e^{-\lambda (1-x)}\lambda x\bigl(1-e^{-\lambda x}\bigr)\\
    &\qquad \quad\quad + \P[X_0\geq 1, X_1=1] + \P[X_0=1,X_1\geq 1]\\
    &= \lambda\bigl( e^{-\lambda x}+e^{-\lambda(1-x)} - 2e^{-\lambda}\bigr)
\end{align*}
Substituting from \eqref{psi_one_0_one_1}, we get
\begin{equation*}\label{simplified_ell_1}
  \E[\ell_{1}(\Tpiv)] = \lambda \bigl(1 - x(\lambda) - e^{-\lambda}\bigr).
\end{equation*}
For $\lambda=\lambda_c$, we have $x(\lambda)=0$ and $\E[\ell_1(\Tpiv)]=\lambda_0(1-e^{-\lambda_0})=1$.
With some laborious calculus, one can establish that as $\lambda$ increases, the quantity
$\E[\ell_1(\Tpiv)]$ decreases and then increases, reaching $1$ at $\lambda_1\approx 2.30$.
Thus, by Theorem~\ref{main 2 colours} and Proposition~\ref{prop:2state.regularity}\ref{i:eig}, 
this fixed point $x(\lambda)$ is interpretable for $\lambda\in[\lambda_0,\lambda_1]$
and rogue for $\lambda>\lambda_1$. A plot showing the behaviour
of the fixed points is given in Figure~\ref{fig:A0001}.
\end{example}

\begin{example}\label{ex.manyroots}
Finally we present an example to demonstrate that the automaton may have many fixed points. Consider the automaton $A$ on colour set $\Sigma = \{0, 1\}$, where a node is in state $1$ unless it has $x$ children in state $1$ for $x\in \{1,2,3,4,5\}\cup\{8,9,10,11\}$ in which case it is state $0$. That is, the automaton is the map $(n_0,n_1)\mapsto\1\{(n_1\in\{0,6,7\})\lor(n_1\geq 12)\}$.  The plot of the fixed points for this automaton with child distribution $\Poi(\lambda)$ is shown in Figure~\ref{fig:manyroots}.

\begin{figure}
\centering
    \begin{tikzpicture}[xscale=1.1,yscale=4.21]
      \draw (16,0) -- (16,1);
      \foreach \x in {0, .2, .4, .6, .8, 1}
 \draw (16.1, \x)--(16,\x) node[left,font=\small] {$\x$};
      \foreach \x in {16,17, 18, 19, 20, 21, 22, 23, 24, 25}
        \draw (\x, .03636)--(\x, 0) node[below,font=\small] {$\x$};
      \path (15.2, .5) node {$p$}
            (20.5, -.15) node {$\lambda$};
      \draw (16,0) -- (26,0);
      \draw[thick,blue] plot file{many_glist.table};
      \draw[thick,red] plot file{many_rlistA.table};
      \draw[thick,red] plot file{many_rlistB.table};
      \draw[thick,red] plot file{many_rlistC.table};
      \draw[thick,red] plot file{many_rlistD.table};
      \draw[thick,red] plot file{many_rlistE.table};

      \draw[very thick,blue] (25,.6)--(25.7,.6) node[black,anchor=mid west] {interpretable};
      \draw[very thick,red] (25,.4)--(25.7,.4) node[black,anchor=mid west] {rogue};

    \end{tikzpicture}
    \vspace{-5mm}
\caption{A plot of the fixed points for the automaton in Example~\ref{ex.manyroots} with $\Poi(\lambda)$ child distribution.}
\label{fig:manyroots}
\end{figure}
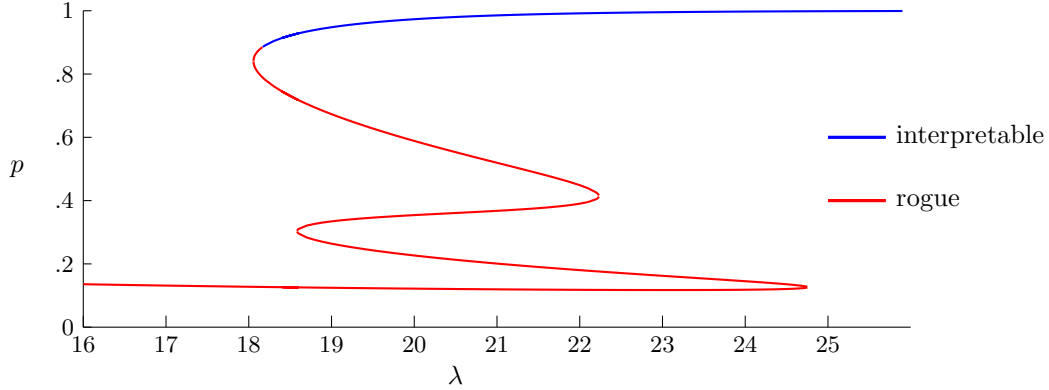

\end{example}

\subsection{First-order interpretations}\label{sec:first.order}

As we mentioned in the introduction, the papers \cite{PS1,PS2} investigated tree automata
and fixed points corresponding to statements of first-order logic.
The goal of \cite{PS2} is to study the probability that $T\sim\GW(\Poi(\lambda))$
satisfies some given first-order sentence of quantifier depth $k$.
Recall from Section~\ref{sec:connections} that there is an automaton on the set
of rank~$k$ types and an interpretation
given by sending a tree to its type. Assuming that the child distribution
is $\Poi(\lambda)$, the automaton distribution
map for the tree automaton is then shown to be a contraction \cite[Theorem~3.2]{PS2}, which
implies that it has a unique fixed point. This fixed point is also shown to be a smooth function
of $\lambda$ \cite[Theorem~2.4]{PS2}. Finally, this is applied to the original problem: since
the set of trees satisfying a given first-order sentence $\varphi$ of quantifier depth $k$ is the union
of a collection of rank~$k$ types, the probability that $T$ satisfies $\varphi$ is also a smooth
function of $\lambda$.

All of this work was done with no explicit mention of the concept of interpretations.
Our goal here is to put it more comfortably into this paper's framework.
We call $\iota\colon\Tt\to\Sigma$ a \emph{first-order interpretation} if each set
of trees $\{t\in\Tt\colon\iota(t)=\sigma\}$ for $\sigma\in\Sigma$ can be defined in the 
first-order language described in Section~\ref{sec:connections}.

To avoid reproving results in \cite{PS1,PS2}, we continue to assume that $\chi\sim\Poi(\lambda)$,
but we expect that the results should hold for general child distributions.
As usual, the assumption that a fixed point has no zero entries causes no loss of generality,
since the set $\Sigma$ can be shrunk and the automaton considered as one on a smaller set of states.

\begin{thm}\label{thm:first.order}
  Assume that $\chi\sim\Poi(\lambda)$, and let $\Sigma$ be any finite set of states.
  Let $\iota\colon\Tt\to\Sigma$ be an interpretation
  of an automaton $A$ corresponding to a fixed point $\vec\nu$, which we assume has strictly positive
  entries. If $\iota$ is a first-order interpretation, then $\vec\nu$ is the only fixed point
  of the automaton distribution map.
\end{thm}

\begin{proof}
  Let $\Ttaut_n\subseteq\Tt$ consist of all trees $t$ on which $\iota(t)$ is tautologically determined
  by $t|_n$. That is, $\Ttaut_n$ consists of all trees $t$  
  such that $\iota(t)=\iota(t')$
  for all $t'\in[t]_n$. It follows from \cite[Lemma~5.6]{PS2} that 
  \begin{align}\label{eq:taut.det}
    \lim_{n\to\infty}\P[T\in \Ttaut_n]= 1.
  \end{align}
  
  Let $(t,\tau)$ be an arbitrary tree whose colouring is compatible with the automaton $A$.
  We claim that if $t\in\Ttaut_n$, then $\tau(R_t)=\iota(t)$.
  Indeed, condition on $T\in[t]_n$. Under this conditioning, the vector $\bigl(\iota(T(v))\bigr)_{v\in L_n(t)}$
  is i.i.d.\ $\vec\nu$. Since $\vec\nu$ has strictly positive entries, this vector takes on each
  value in $\Sigma^{\ell_n(t)}$ with positive probability.
  Since $t\in\Ttaut_n$, we have $\iota(T)=\iota(t)$ a.s.
  But $\iota(T)$ is given by iteratively applying the automaton to $\bigl(\iota(T(v))\bigr)_{v\in L_n(t)}$,
  from which we can conclude that applying the automaton in this way to \emph{any} vector
  in $\Sigma^{\ell_n(t)}$ yields $\iota(t)$. Thus, since $\tau(R_t)$ is given by applying
  the automaton to $\bigl(\iota(t(v))\bigr)_{v\in L_n(t)}$, it too is equal to $\iota(t)$.
  
  Now, suppose that $\vec\nu'$ is another fixed point of the automaton map, and let $(T,\omega)$
  be the random state tree associated with $\vec\nu'$. 
  If $T\in\Ttaut_n$ for any $n$, then $\omega(R_T)=\iota(T)$ by the claim we have just proved.
  By \eqref{eq:taut.det}, it holds with probability one that $T\in\Ttaut_n$ for some value of $n$
  (observe that $T\in\Ttaut_n$ forms an increasing sequence of events).
  Hence $\omega(R_T)=\iota(T)$ a.s.
  Thus $\omega(R_T)\sim\vec\nu$, since $\iota(T)\sim\vec\nu$. But by the definition of the random
  state tree, $\omega(R_T)\sim\vec\nu'$, demonstrating that $\vec\nu=\vec\nu'$.
\end{proof}

\section{Further questions}\label{sec:open}

The most straightforward open problem is to extend Theorem~\ref{main 2 colours} 
to $3\leq \abs{\Sigma}<\infty$.
Theorem~\ref{thm:k.state} already provides one direction of the theorem, leaving
the critical and supercritical cases. As we discussed in Remarks~\ref{rmk:critical.case}
and \ref{rmk:kstate.supercritical}, it might be possible to adapt the two-state proofs.
In both cases, the difficulty is that the pivot tree need not be positive regular. In fact,
it seems to us that when the pivot tree is positive regular, all proofs go through as is,
and Theorem~\ref{main 2 colours} holds in general for $\abs{\Sigma}<\infty$ (though we have not
checked every detail).

Beyond this, two generalizations interest us. First, extending the results to
infinite state spaces would allow the theory
to address situations like those considered in \cite{MS}.  Second, one could consider
randomized automata: give
each vertex~$v$ an independent source of randomness $X_v$ and then allow the automaton
to determine the state of a vertex from the states of its children together with $X_v$.
This situation often arises in practice and is the model considered in \cite{AB}.
Extending the theory to this case might yield answers to questions about endogeny,
as discussed in Section~\ref{sec:connections}.

In a different direction, we wonder what configurations
of fixed points are possible.
For example, when $3\leq\abs{\Sigma}<\infty$, can an automaton have infinitely many fixed
points? (This can be ruled out when $\abs{\Sigma}=2$ by arguing that the automaton distributional
map is analytic.) In the case $\abs{\Sigma}=2$, for any specified finite set of rogue and
interpretable fixed points, is there an automaton and a child distribution to match them?
Even restricting ourselves to two-state monotone automata, it is not clear which sets of rogue
and interpretable fixed points can occur. 

Section~\ref{sec:first.order} also raises some questions.
For example, Theorem~\ref{thm:first.order} provides a condition on an 
interpretation that makes the corresponding tree automaton
have a unique fixed point. This suggests
the problem of giving conditions on the tree automaton itself that force its 
automaton distribution map to have a unique fixed point.

\section*{Acknowledgments}

We thank Joel Spencer, who set this work in motion and generously advised us.
We also thank Leonid Libkin for guiding us through the literature on tree automata
in logic.

\bibliography{pivotbibfile}

\providecommand{\bysame}{\leavevmode\hbox to3em{\hrulefill}\thinspace}
\providecommand{\MR}{\relax\ifhmode\unskip\space\fi MR }
\providecommand{\MRhref}[2]{%
  \href{http://www.ams.org/mathscinet-getitem?mr=#1}{#2}
}
\providecommand{\href}[2]{#2}
\begin{thebibliography}{10}

\bibitem{AB}
David~J. Aldous and Antar Bandyopadhyay, \emph{A survey of max-type recursive
  distributional equations}, Ann. Appl. Probab. \textbf{15} (2005), no.~2,
  1047--1110. \MR{2134098 (2007e:60010)}

\bibitem{Endog3}
Gerold Alsmeyer and Matthias Meiners, \emph{Fixed points of the smoothing
  transform: two-sided solutions}, Probab. Theory Related Fields \textbf{155}
  (2013), no.~1-2, 165--199. \MR{3010396}

\bibitem{athreya-ney}
Krishna~B. Athreya and Peter~E. Ney, \emph{Branching processes},
  Springer-Verlag, New York-Heidelberg, 1972, Die Grundlehren der
  mathematischen Wissenschaften, Band 196. \MR{0373040}

\bibitem{Endog2}
Antar Bandyopadhyay, \emph{Endogeny for the logistic recursive distributional
  equation}, Z. Anal. Anwend. \textbf{30} (2011), no.~2, 237--251. \MR{2793003}

\bibitem{BKKKL}
Jean Bourgain, Jeff Kahn, Gil Kalai, Yitzhak Katznelson, and Nathan Linial,
  \emph{The influence of variables in product spaces}, Israel J. Math.
  \textbf{77} (1992), no.~1-2, 55--64. \MR{1194785}

\bibitem{BDF}
Nicolas Broutin, Luc Devroye, and Nicolas Fraiman, \emph{Recursive functions on
  conditional {G}alton--{W}atson trees}, available at arXiv:1805.09425, 2018.

\bibitem{automata_1}
Hubert Comon, Max Dauchet, R{\'e}mi Gilleron, Christof L{\"o}ding, Florent
  Jacquemard, Denis Lugiez, Sophie Tison, and Marc Tommasi, \emph{Tree automata
  techniques and applications}, available at
  \url{http://www.grappa.univ-lille3.fr/tata}, 2007.

\bibitem{Dekking}
F.~M. Dekking, \emph{Branching processes that grow faster than binary
  splitting}, Amer. Math. Monthly \textbf{98} (1991), no.~8, 728--731.
  \MR{1130682}

\bibitem{Durrett}
Rick Durrett, \emph{Probability: theory and examples}, 4.1 ed., April 21, 2013,
  available at \url{http://services.math.duke.edu/~rtd/PTE/pte.html}. Fourth
  edition published by Cambridge University Press in 2010.

\bibitem{FK}
Ehud Friedgut and Gil Kalai, \emph{Every monotone graph property has a sharp
  threshold}, Proc. Amer. Math. Soc. \textbf{124} (1996), no.~10, 2993--3002.
  \MR{1371123}

\bibitem{GS}
Christophe Garban and Jeffrey~E. Steif, \emph{Noise sensitivity of {B}oolean
  functions and percolation}, Institute of Mathematical Statistics Textbooks,
  Cambridge University Press, New York, 2015. \MR{3468568}

\bibitem{Harris}
Theodore~E. Harris, \emph{The theory of branching processes}, Die Grundlehren
  der Mathematischen Wissenschaften, Bd. 119, Springer-Verlag, Berlin;
  Prentice-Hall, Inc., Englewood Cliffs, N.J., 1963, Also available as RAND
  report R-381-PR, \url{https://www.rand.org/pubs/reports/R381.html}.
  \MR{0163361}

\bibitem{Endog1}
Saul Jacka and Marcus Sheehan, \emph{The noisy veto-voter model: a recursive
  distributional equation on {$[0,1]$}}, J. Appl. Probab. \textbf{45} (2008),
  no.~3, 670--688. \MR{2455177}

\bibitem{KKL}
Jeff Kahn, Gil Kalai, and Nathan Linial, \emph{The influences of variables on
  {B}oolean functions}, 29th Annual Symposium on Foundations of Computer
  Science (White Plains, 1988), IEEE Comput. Soc. Press, Washington, D.C.,
  1988, pp.~68--80.

\bibitem{Kallenberg}
Olav Kallenberg, \emph{Foundations of modern probability}, second ed.,
  Probability and its Applications (New York), Springer-Verlag, New York, 2002.
  \MR{1876169}

\bibitem{Endog4}
Victor Kleptsyn and Michele Triestino, \emph{Cut-off method for endogeny of
  recursive tree processes}, available at arXiv:1610.06946, 2016.

\bibitem{Libkin}
Leonid Libkin, \emph{Elements of finite model theory}, Texts in Theoretical
  Computer Science. An EATCS Series, Springer-Verlag, Berlin, 2004.
  \MR{2102513}

\bibitem{LPP}
Russell Lyons, Robin Pemantle, and Yuval Peres, \emph{Conceptual proofs of
  {$L\log L$} criteria for mean behavior of branching processes}, Ann. Probab.
  \textbf{23} (1995), no.~3, 1125--1138. \MR{1349164}

\bibitem{LP}
Russell Lyons and Yuval Peres, \emph{Probability on trees and networks},
  Cambridge Series in Statistical and Probabilistic Mathematics, vol.~42,
  Cambridge University Press, New York, 2016. \MR{3616205}

\bibitem{MS}
James~B. Martin and Roman Stasi\'nski, \emph{Minimax functions on
  {G}alton--{W}atson trees}, available at arXiv:1806.07838, 2018.

\bibitem{O'Donnell}
Ryan O'Donnell, \emph{Analysis of {B}oolean functions}, Cambridge University
  Press, New York, 2014. \MR{3443800}

\bibitem{PakesDekking}
Anthony~G. Pakes and F.~M. Dekking, \emph{On family trees and subtrees of
  simple branching processes}, J. Theoret. Probab. \textbf{4} (1991), no.~2,
  353--369. \MR{1100239}

\bibitem{PSW}
Boris Pittel, Joel Spencer, and Nicholas Wormald, \emph{Sudden emergence of a
  giant {$k$}-core in a random graph}, J. Combin. Theory Ser. B \textbf{67}
  (1996), no.~1, 111--151. \MR{1385386}

\bibitem{PS1}
Moumanti Podder and Joel Spencer, \emph{First order probabilities for
  {G}alton-{W}atson trees}, A journey through discrete mathematics, Springer,
  Cham, 2017, pp.~711--734. \MR{3726620}

\bibitem{PS2}
\bysame, \emph{Galton-{W}atson probability contraction}, Electron. Commun.
  Probab. \textbf{22} (2017), Paper No. 20, 16. \MR{3627009}

\bibitem{Sevastyanov}
B.~A. Sevast'yanov, \emph{On the theory of branching random processes}, Doklady
  Akad. Nauk SSSR (N.S.) \textbf{59} (1948), 1407--1410. \MR{0024090}

\bibitem{Thomas}
Wolfgang Thomas, \emph{Languages, automata, and logic}, Handbook of formal
  languages, {V}ol.\ 3, Springer, Berlin, 1997, pp.~389--455. \MR{1470024}

\end{thebibliography}

\end{document}